\documentclass[preprint,3p, 12pt, times]{elsarticle}
\usepackage[T1]{fontenc}
\usepackage{lmodern}
\usepackage{microtype}
\usepackage{amsmath,amssymb,amsthm,mathtools}
\usepackage{tikz-cd}
\usepackage{enumitem}
\usepackage[unicode,colorlinks=true,linkcolor=blue!45!black,citecolor=blue!45!black,urlcolor=blue!45!black]{hyperref}
\hypersetup{pdftitle={The reflective hull of the two-element chain in DCPO: properness, maximal Gamma-faithfulness, and an internal reflection formula},pdfauthor={Xulong He, Zhenchao Lyu, Yuxu Chen, Hui Kou},pdfkeywords={directed-complete partial order, reflective subcategory, equability, Scott topology, Gamma-faithfulness, order sobrification}}
\usepackage[nameinlink,capitalise,noabbrev]{cleveref}

\journal{Topology and its Applications}
\biboptions{numbers,sort&compress}

\allowdisplaybreaks
\numberwithin{equation}{section}

\usepackage{aliascnt}

\newtheorem{theorem}{Theorem}[section]

\newaliascnt{lemma}{theorem}
\newtheorem{lemma}[lemma]{Lemma}
\aliascntresetthe{lemma}

\newaliascnt{proposition}{theorem}
\newtheorem{proposition}[proposition]{Proposition}
\aliascntresetthe{proposition}

\newaliascnt{corollary}{theorem}
\newtheorem{corollary}[corollary]{Corollary}
\aliascntresetthe{corollary}

\newaliascnt{claim}{theorem}

\aliascntresetthe{claim}

\newaliascnt{hypothesis}{theorem}

\aliascntresetthe{hypothesis}

\theoremstyle{definition}

\newaliascnt{definition}{theorem}
\newtheorem{definition}[definition]{Definition}
\aliascntresetthe{definition}

\newaliascnt{example}{theorem}

\aliascntresetthe{example}

\theoremstyle{remark}

\newaliascnt{remark}{theorem}
\newtheorem{remark}[remark]{Remark}
\aliascntresetthe{remark}

\usepackage[nameinlink,capitalise,noabbrev]{cleveref}

\newcommand{\DCPO}{\mathbf{DCPO}}
\newcommand{\Rtwo}{\mathcal R_2}

\newcommand{\Irr}{\operatorname{Irr}}
\newcommand{\ReflHull}{\operatorname{ReflHull}}

\newcommand{\Eq}{\operatorname{Eq}}

\newcommand{\cl}{\operatorname{cl}}
\newcommand{\id}
{\operatorname{id}}

\newcommand{\Supp}{\operatorname{Supp}}

\newcommand{\down}{\mathord{{\downarrow}}}
\newcommand{\WDomi}{\mathbf{WDomi}}
\newcommand{\DOMI}{\mathbf{DOMI}}
\newcommand{\two}{2}
\newcommand{\N}{\mathbb{N}}
\newcommand{\J}{\mathbb{J}}
\newcommand{\LJ}{L\mathbb{J}}

\newcommand{\Gm}{\Gamma}

\begin{document}

\begin{frontmatter}

\title{The reflective hull of the two-element chain in \(\DCPO\): properness, maximal \(\Gamma\)-faithfulness, and an internal reflection formula}

\author[scu]{Xulong He}
\ead{hexulong@scu.edu.cn}
\author[scu]{Zhenchao Lyu}
\ead{zhenchaolyu@scu.edu.cn}

\author[scu]{Yuxu Chen}
\ead{chenyuxu@scu.edu.cn}

\author[scu]{Hui Kou}
\ead{kouhui@scu.edu.cn}

\affiliation[scu]{organization={School of Mathematics, Sichuan University},
  city={Chengdu},
  country={China}}

\begin{abstract}
Let \(\DCPO\) be the category of dcpos and Scott-continuous maps, and let \(\Rtwo\) be the reflective hull of the two-element chain \(2\). 
We prove that \(\Rtwo\) is the least non-discrete proper reflective full subcategory of \(\DCPO\). 
It is closed under limits and Skula-closed sub-dcpos and contains every sober dcpo. 
We show that \(\Rtwo\) is the maximal \(\Gamma\)-faithful full
subcategory of \(\DCPO\) that strictly contains weakly dominated dcpos. 
Finally, we give the concrete construction of \(2\)-reflection internally by a transfinite iteration and a criterion for reflective full
subcategories of \(\DCPO\), analogous to the Keimel--Lawson conditions for
\(T_0\) topological spaces.
\end{abstract}


\end{frontmatter}

\section{Introduction}

In domain theory and non-Hausdorff topology, d-spaces, well-filtered spaces, and sober spaces are three important classes of topological spaces \cite{gierz2003,goubault2013}. With continuous maps as morphisms, let $\mathbf{Top_0}$, $\mathbf{Sob}$, $\mathbf{Top_d}$, and $\mathbf{Top_w}$ denote the categories of $T_0$ spaces, sober spaces, d-spaces, and well-filtered spaces, respectively. These are full subcategories of $\mathbf{Top_0}$.
It is well known that every $T_0$ topological space admits a sobrification \cite{goubault2013}. Consequently, the category $\mathbf{Sob}$ of sober spaces is a reflective subcategory of $\mathbf{Top_0}$. Using the d-closure, Wyler proved in 1981 that $\mathbf{Top_d}$ is reflective in $\mathbf{Top_0}$ \cite{Wyler1981}. In 1999, Ershov observed in \cite{ershov1999} that the d-completion of a space $X$, namely its d-reflection, can also be obtained by adjoining to $X$ the closures of directed subsets and then repeating this procedure by transfinite induction.

In 2009, Keimel and Lawson \cite{Keimel2009} proved that if \(\mathcal K\) is a full subcategory of $\mathbf{Top_0}$ containing $\mathbf{Sob}$ and satisfying certain conditions, subsequently known as the Keimel--Lawson conditions, then \(\mathcal K\) is a reflective subcategory of $\mathbf{Top_0}$. In 2019, Wu, Xi, Xu, and Zhao applied the Keimel--Lawson approach to prove that $\mathbf{Top_w}$ is reflective in $\mathbf{Top_0}$ \cite{Wu2019}. In the same year, Shen, Xi, Xu, and Zhao adapted Ershov's construction of the d-completion of a $T_0$ space to obtain a construction of the well-filtered reflection of a $T_0$ space \cite{Shen2019}. In subsequent work, Xu provided a direct method for constructing the \(\mathcal K\)-reflection of a $T_0$ space \cite{Xu2020}.
Ershov \cite{Ershov2022} and Shen, Xi, and Zhao \cite{Shen2024} proved that if a full subcategory \(\mathcal K\) of $\mathbf{Top_0}$ contains a non-$T_1$ space, then \(\mathcal K\) is reflective if and only if it satisfies the Keimel--Lawson conditions. This result shows that, in a certain sense, the Keimel--Lawson conditions are both necessary and sufficient for reflectivity. The studies described above provide a range of methods and techniques for constructing reflection functors in the category of $T_0$ topological spaces. In particular, the categories of sober spaces, d-spaces, and well-filtered spaces are all reflective subcategories of $\mathbf{Top_0}$.
\par

Directed-complete partial orders are among the basic objects of domain theory. Directed suprema model limits of increasing computations, while the Scott topology is the canonical topology compatible with those suprema \cite{sc-4}. The category \(\DCPO\) of dcpos and Scott-continuous maps is complete, cocomplete, and cartesian closed: for dcpos \(A,B\), the hom-object \([A\to B]\) is the dcpo of Scott-continuous maps ordered pointwise; see, for example, \cite{gierz2003,goubault2013}.
In their study of observationally induced algebras, Battenfeld, Keimel, and Streicher introduced repletion in a locally dcpo-enriched setting \cite{bks2014}. Their definition requires restriction between function dcpos to be an isomorphism, equivalently, the unique-extension operator to preserve the pointwise order.
The two-element chain \(\Sigma=2=\{0<1\}\) plays the role of the two-element chain in \(\DCPO\). Moreover, if a reflective full subcategory contains a non-discrete dcpo, then it contains \(2\) as a retract. This motivates the study of the reflective hull of \(2\).

The main purpose of this paper is to determine the categorical and order-topological position of \(\Rtwo=\ReflHull_{\DCPO}(2)\) ($R_2$ is the full subcategory of all \(2\)-complete dcpos which is a reflective subcategory of \(\DCPO\) \cite{bks2014} and $\mathrm{ReflHull}_{\mathbf{DCPO}}(R)$ denotes the least reflective
full subcategory of $\mathbf{DCPO}$ containing \(R\)). We prove that it is minimal and maximal in two different senses: it is the least non-discrete reflective full subcategory of \(\DCPO\), and it is maximal among full subcategories on which the Scott-closed-set lattice determines objects up to isomorphism. The latter result is closely related to the Ho--Zhao problem, which asks
whether \(\DCPO\) is $\Gamma$-faithful, that is, whether
\[
    \Gamma(D) \cong \Gamma(E)
    \ \Longrightarrow\ 
    D \cong E,
\]
for arbitrary dcpos $D$ and $E$, where
$\Gamma(D)$ denotes the lattice of Scott-closed subsets of $D$ with inclusion order. Ho, Goubault-Larrecq, Jung, and Xi
answered this problem negatively in~\cite{ho2018}
by constructing non-isomorphic dcpos with isomorphic lattices of
Scott-closed sets. Their construction will also be used to show
that $\mathcal{R}_2$ is a proper subcategory of \(\DCPO\).
In contrast to the negative answer for the whole category \(\DCPO\), we show
that $\mathcal{R}_2$ is $\Gamma$-faithful and, moreover, maximal
among $\Gamma$-faithful full subcategories of \(\DCPO\). Thus $\mathcal{R}_2$
provides a canonical maximal subclass of \(\DCPO\) on which the Scott-closed-set
lattice determines dcpos up to isomorphism. We also construct the \(2\)-reflection internally in the order sobrification and prove a relative hull criterion over \(\Rtwo\).

A further comparison concerns the weakly dominated dcpos of
Miao--Hou--Jia--Li \cite{miao2024}. We prove that every weakly dominated dcpo belongs
to \(\Rtwo\), and also that every \(\Gamma\)-unique dcpo (a dcpo \(P\) is called \emph{\(\Gamma\)-unique}
if, for every dcpo \(Q\),
    $\Gamma(P)\cong \Gamma(Q)
    \Longrightarrow
    P\cong Q$) belongs
to \(\Rtwo\). Applying the latter result to the stratified Johnstone dcpo \(\LJ\) of Xu and Zhao, and using the structural properties established in \cite{XuZhao2019}, we obtain \(\LJ\in\Rtwo\setminus\WDomi\), and hence \(\WDomi\subsetneq\Rtwo\), where \(\WDomi\) is the full subcategory of weakly dominated dcpos.

We also construct the \(2\)-reflection internally in the order sobrification and prove a relative hull criterion over \(\Rtwo\). We further give an internal description of the \(2\)-reflection.
Starting from the order sobrification of a dcpo \(D\), we construct the
reflection by a transfinite pruning process which removes precisely the
points not forced by their traces on \(D\).  The stable object obtained
in this way is shown to be canonically isomorphic to \(\mathbf{\mathbf{L_2}}D\).  This
provides a concrete order-theoretic formula for the reflector and
clarifies how the points added by the \(2\)-reflection arise.

Finally, we establish a Keimel--Lawson type criterion for reflective
full subcategories of \(\DCPO\).  Since every non-discrete reflective
full subcategory contains \(\Rtwo\), the criterion describes further
reflections by taking suitable sub-dcpos of the \(2\)-reflection.  This
gives a domain-theoretic analogue of the Keimel--Lawson characterization
for \(T_0\) spaces.

The paper is organized as follows.
In \cref{sec:preliminaries}, we recall the basic notions of reflective
subcategories, equability which is the reflective subcategory of \(\DCPO\) defined by \cite[Definition~2.2]{bks2014}, Scott topology, and order sobrification.
We prove that the reflective hull of a dcpo \(R\) can be characterized by the
class of \(R\)-replete dcpos, which provides the main tool for studying
\(\Rtwo\).
In \cref{sec:proper}, we investigate the categorical position of
\(\Rtwo\). We prove that \(\Rtwo\) is a proper reflective full subcategory of
\(\DCPO\) and is the least reflective full subcategory containing a
non-discrete dcpo. The Ho--Zhao example is used to show that this inclusion
is strict.
In \cref{sec:closure}, we study closure properties of \(\Rtwo\). We show
that \(\Rtwo\) is closed under limits and Skula-closed sub-dcpos. As
applications, we prove that specialization orders of sober spaces,
Scott-sober dcpos, order sobrifications, complete lattices, and
bounded-complete dcpos all belong to \(\Rtwo\).
In \cref{sec:gamma}, we study the reconstruction problem determined by
Scott-closed-set lattices. We prove that, although \(\DCPO\) is not
\(\Gamma\)-faithful in general, the subcategory \(\Rtwo\) is
\(\Gamma\)-faithful and maximal among the \(\Gamma\)-faithful full
subcategories of \(\DCPO\).
In \cref{sec:weakdom}, we compare \(\Rtwo\) with weakly dominated dcpos. We
prove that every weakly dominated dcpo belongs to \(\Rtwo\), and use the
stratified Johnstone dcpo to establish the strict inclusion
$\mathrm{WDomi}\subsetneq \Rtwo$.
In \cref{sec:internal-reflection}, we give an explicit construction of the
\(2\)-reflection inside the order sobrification by means of a transfinite
iteration. This provides an internal description of the reflection process
and identifies the points added by the reflection.
Finally, in \cref{sec:KL}, we establish a criterion for reflective full
subcategories of \(\DCPO\), analogous to the Keimel--Lawson conditions for
\(T_0\) topological spaces. This gives a characterization of reflectivity in
\(\DCPO\) and extends the role of Keimel--Lawson type conditions from
topological spaces to the domain-theoretic setting.

\section{Preliminaries and equability}\label{sec:preliminaries}

We assume familiarity with the basic knowledge of order and topology as laid out in \cite{gierz2003}
or \cite{goubault2013}. Throughout the paper, directed subsets are nonempty, \(\N=\{1,2,\ldots\}\), and every full subcategory is understood to be closed under isomorphisms. For dcpos \(A,B\), the dcpo of Scott-continuous maps ordered pointwise is denoted by \([A\to B]\). We write \(\Sigma D=(D,\sigma(D))\) for the Scott space of a dcpo \(D\), \(\sigma(D)\) is the Scott topology of $D$, and \(\Omega X\) for the specialization order of a \(T_0\) space \(X\).

The purpose of this section is to relate reflective hulls to equability.  The main result is \cref{thm:replete-hull}, which identifies
$\operatorname{ReflHull}_{\mathbf{DCPO}}(R)$ with the full subcategory
of $R$-complete dcpos.  This characterization will be used throughout
the paper, mainly for $R=2$ (2 is the two-element chain).

\subsection{Equability and observationally induced reflection}

\begin{definition}
\label{def:reflection}
Let \(\mathcal K\subseteq\DCPO\) be a full subcategory.
We say that \(\mathcal K\) is \emph{reflective} if, for every dcpo \(D\),
there exist \(\mathbf{L}D\in\mathcal K\) and a Scott-continuous map
\[
  \kappa_D:D\longrightarrow \mathbf{L}D
\]
such that, for every \(K\in\mathcal K\) and every Scott-continuous map
\(f:D\to K\), there exists a unique Scott-continuous map
\(\widetilde f:\mathbf{L}D\to K\) for which the diagram
\[
\begin{tikzcd}[column sep=large,row sep=large]
D \arrow[r,"\kappa_D"] \arrow[dr,"f"']
  & \mathbf{L}D \arrow[d,dashed,"{!\widetilde f}"] \\
  & K
\end{tikzcd}
\]
commutes; that is,
\[
  f=\widetilde f\circ\kappa_D.
\]
The map \(\kappa_D\) is called the \(\mathcal K\)-reflection unit of \(D\),
and \(\mathbf{L}D\) is called the \(\mathcal K\)-reflection of \(D\).

Equivalently, for every \(K\in\mathcal K\), composition with
\(\kappa_D\) induces a bijection
\[
  \mathcal K(\mathbf{L}D,K)
  \longrightarrow
  \DCPO(D,K),
  \ 
  g\longmapsto g\circ\kappa_D.
\]

We shall also use the following stronger notion. A reflection is called enriched if this bijection is an order isomorphism.

Unless explicitly stated otherwise, ``reflective'' means ordinary reflective. We use ``enriched reflective'' when the induced hom-set bijections are order isomorphisms. In fact, \cref{rem:2.7} can illustrate that every reflective full subcategory of \(\DCPO\) is automatically enriched reflective.

\end{definition}

\begin{lemma}\label{lem:reflective-closure} (1) 
{\cite[Exercise 4.12.8]{goubault2013}}
\(\DCPO\) is a complete category, whose products are given by Cartesian products
equipped with the pointwise order.

More precisely, let $(D_i)_{i\in I}$ be an arbitrary family of dcpos. 
Its product is the Cartesian product $\prod_{i\in I}D_i$ equipped with the 
product order, where
\[
x \leq y \ \Longleftrightarrow\  x_i \leq_i y_i
\]
for every $i\in I$.

In \(\DCPO\), the equalizer of two morphisms 
$g_1,g_2:D_1\to D_2$ is the subset $[g_1=g_2]$ of $D_1$,
defined by
\[
[g_1=g_2]=\{z\in D_1\mid g_1(z)=g_2(z)\}.
\]
It inherits the suborder induced from $D_1$, together with the canonical
inclusion map
\[
\pi:[g_1=g_2]\to D_1 .
\]

\item (2) {\cite[Lemma~5.6.6]{Riehl2017}} A reflective full subcategory of \(\DCPO\) is closed under retracts and under limits computed in \(\DCPO\).
\end{lemma}

\begin{definition}[{\cite[Definition~2.1]{bks2014}}]
Let \(e:A\to B\) be a Scott-continuous map and let \(C\) be a dcpo.
We say that \(e\) is \emph{\(C\)-equable} if the precomposition map
\[
e^*:[B\to C]\longrightarrow[A\to C],
\ 
g\longmapsto g\circ e,
\]
is an isomorphism of dcpos.
\end{definition}

\begin{definition}[{\cite[Definition~2.2]{bks2014}}]
Fix a dcpo \(R\), and define
\[
  E_R=
  \left\{e:A\to B\ \middle|\
  e^*:[B\to R]\cong[A\to R]
  \text{ in }\DCPO
  \right\}.
\]
A dcpo \(C\) is called \emph{\(R\)-complete} or \emph{\(R\)-replete} if every morphism in \(E_R\) is \(C\)-equable. The full subcategory of all \(R\)-complete dcpos is denoted by \(E_R^{\perp}\). 

In particular, when \(R=2\), we write
\(
\mathcal{R}_2=E_2^{\perp}
\)
for the reflective hull of the two-element chain \(2\) in \(\DCPO\) after \cref{sec:proper}.
\end{definition}

The following is the empty-signature specialization of \cite[Theorem~2.3 and Proposition~2.7]{bks2014}.

\begin{theorem}[{\cite[Theorem~2.3 and Proposition~2.7]{bks2014}}]\label{thm:BKS}
For every dcpo \(R\), the \(R\)-complete dcpos form a reflective full subcategory of \(\DCPO\). The reflection unit
\(
  r_D:D\longrightarrow R(D)
\)
is \(R\)-equable.
\end{theorem}

In the sequel \(\ReflHull_{\DCPO}(R)\) means the least reflective full subcategory containing \(R\). The argument below also shows
that the resulting reflection is enriched.

\begin{lemma}
\label{lem:retract}
If \(D\) is a non-discrete dcpo, then \(2\) is a retract of \(D\).
\end{lemma}

\begin{proof}
Choose \(a<b\) in \(D\). Since the Scott topology is \(T_0\), there is a Scott-open set \(U\) with \(b\in U\) and \(a\notin U\). Define \(i:2\to D\) by \(i(0)=a\), \(i(1)=b\), and let \(p=\chi_U:D\to 2\). Both maps are Scott continuous and \(p\circ i=\id_{2}\).
\end{proof}

\begin{figure}[ht]
\centering
\begin{tikzcd}
2
  \arrow[r,shift left=.7ex,"i"]
&
D
  \arrow[l,shift left=.7ex,"p"]
\end{tikzcd}
\ 
$p\circ i=\id_2$.
\caption{The two-element chain as a retract of a non-discrete dcpo.}
\label{fig:retract}
\end{figure}

\begin{theorem}\label{thm:replete-hull}
For every dcpo \(R\), $\mathrm{ReflHull}_{\mathbf{DCPO}}(R)$ denotes the least reflective
full subcategory of $\mathbf{DCPO}$ containing \(R\).
\[
  \ReflHull_{\DCPO}(R)=E_R^{\perp}.
\]
In other words, the \(R\)-complete dcpos form the least reflective full subcategory of \(\DCPO\) containing \(R\).
\end{theorem}

\begin{proof}
By \cref{thm:BKS}, \(E_R^\perp\) is reflective and contains \(R\), so \(\ReflHull_{\DCPO}(R)\subseteq E_R^\perp\). For the reverse inclusion, let \(\mathcal K\) be reflective with \(R\in\mathcal K\), let \(\kappa_D:D\to \mathbf{K}D\) be the reflection unit, and suppose \(D\in E_R^\perp\). We show that \(\kappa_D\in E_R\).

The map \(\kappa_D^*:[\mathbf{K}D\to R]\to[D\to R]\) is bijective. If \(R\) is discrete, both hom-dcpos have the discrete order, so this bijection is an order isomorphism. Suppose that \(R\) is non-discrete. By \cref{lem:retract,lem:reflective-closure}, \(\two\in\mathcal K\). Hence \(\kappa_D^{-1}:\sigma(\mathbf{K}D)\to\sigma(D)\) is a bijective frame homomorphism and therefore an order isomorphism.

Let \(f\leq g:D\to R\), and denote their unique extensions by \(\bar f,\bar g:\mathbf{K}D\to R\). For each \(U\in\sigma(R)\), the inequality \(\chi_U\circ f\leq\chi_U\circ g\) and the order-preserving extension into \(\two\) give \(\chi_U\circ\bar f\leq\chi_U\circ\bar g\). Thus, for every \(y\in \mathbf{K}D\) ,
\(
\overline{f}(y)\in U
\Longrightarrow
\overline{g}(y)\in U
\). Since the specialization order of the Scott topology of \(R\) is its original order, \(\bar f\leq\bar g\). Thus the inverse of \(\kappa_D^*\) is monotone, so \(\kappa_D^*\) is an isomorphism of dcpos and \(\kappa_D\in E_R\).

Since \(D\in E_R^\perp\), the morphism \(\kappa_D\in E_R\) is \(D\)-equable. In particular, the identity map \(\id_D:D\to D\) has a Scott-continuous extension \(r:\mathbf{K}D\to D\) such that \(r\circ\kappa_D=\id_D\).

Consider the two Scott-continuous maps \(\id_{\mathbf{K}D}\) and \(\kappa_D\circ r:\mathbf{K}D\to \mathbf{K}D\). Their composites with \(\kappa_D\) coincide:
\[
  (\kappa_D\circ r)\circ\kappa_D
  =
  \kappa_D\circ(r\circ\kappa_D)
  =
  \kappa_D
  =
  \id_{\mathbf{K}D}\circ\kappa_D.
\]
Since \(\mathbf{K}D\in\mathcal K\), the reflection property implies that
\(\kappa_D^*:[\mathbf{K}D\to \mathbf{K}D]\to[D\to \mathbf{K}D]\) is injective. Hence \(\kappa_D\circ r=\id_{\mathbf{K}D}\). Thus \(\kappa_D\) is an isomorphism, and consequently \(D\cong \mathbf{K}D\in\mathcal K\).

We have proved that every \(R\)-complete dcpo belongs to every reflective full subcategory of \(\DCPO\) containing \(R\). Therefore
\(
  E_R^\perp\subseteq\ReflHull_{\DCPO}(R).
\)
Combining the two inclusions gives
\(
  \ReflHull_{\DCPO}(R)=E_R^\perp.
\)
\end{proof}

\begin{remark}\label{rem:2.7}

The argument in the proof of \cref{thm:replete-hull} also shows that every reflective
full subcategory of $\mathbf{DCPO}$ is automatically enriched reflective.
Thus, in $\mathbf{DCPO}$, no distinction between ordinary and enriched
reflectivity is needed for full subcategories.
\end{remark}
In particular, for $R=2$, \cref{thm:replete-hull} allows us to study $R_2$
through $2$-equability.  Since Scott-continuous maps into $2$
correspond to Scott-open subsets, this will connect $R_2$ with the
Scott topology.

\section{Properness and minimality}\label{sec:proper}

We now specialize \cref{thm:replete-hull} to \(R=\two\). We first recall the order sobrification, then prove that \(\Rtwo\) is the least non-discrete reflective full subcategory and is properly contained in \(\DCPO\).

\subsection{Sobrification and order sobrification}

For a dcpo \(D\), let \(\Gamma(D)\) denote the complete lattice of
Scott-closed subsets of \(D\), ordered by inclusion. It is a coframe. An
element \(p\ne\varnothing\) of a coframe \(L\) is \emph{irreducible} if
\(p\le a\vee b\) implies \(p\le a\) or \(p\le b\); write \(\Irr(L)\) for
the nonzero irreducible elements, ordered as in \(L\). The enriched reflection at
\(2\) is denoted \(\mathbf{\mathbf{L_2}}D:\DCPO\to\Rtwo\), with unit
\(\lambda_D:D\to \mathbf{\mathbf{L_2}}D\). Following \cite[Section~1]{ho2018}, a full subcategory \(\mathcal C\) is \emph{\(\Gamma\)-faithful} if \(\Gamma(D)\cong\Gamma(E)\) implies \(D\cong E\) for all \(D,E\in\mathcal C\).

\begin{definition}[{\cite[Theorem~8.2.8]{goubault2013}}]
Let \(X\) be a \(T_0\) space. A \emph{sobrification} of \(X\) is a
sober space \(X^s\), together with a continuous map
\(
  \eta_X:X\longrightarrow X^s,
\)
such that, for every sober space \(Y\) and every continuous map
\(f:X\to Y\), there exists a unique continuous map
\(
  f^s:X^s\longrightarrow Y
\)
satisfying
\[
  f^s\circ\eta_X=f.
\]
The map \(\eta_X\) is called a \emph{sobrification map}.
\end{definition}

Let \(\operatorname{Irr}_c(X)\) be the nonempty irreducible closed subsets of \(X\). For \(U\in\mathcal O(X)\), put \(\Diamond U=\{A\in\operatorname{Irr}_c(X):A\cap U\neq\varnothing\}\). These sets form the lower-Vietoris topology on \(\operatorname{Irr}_c(X)\); the resulting space is denoted by \(S(X)\).

\begin{proposition}[{\cite[Theorem~8.2.8]{goubault2013}}]\label{prop:standard-sobrification}
For every \(T_0\) space \(X\), the space \(S(X)\) is sober and \(\eta_X:X\to S(X)\), \(x\mapsto\overline{\{x\}}\), is a sobrification map and a topological embedding.
\end{proposition}

For a dcpo \(D\), the specialization order of \(S(\Sigma D)\) is inclusion. Following \cite[Section~2]{ho2018}, its \emph{order sobrification} is \(\widehat D=\Irr(\Gamma(D))\), with canonical map \(\eta_D:D\to\widehat D\), \(x\mapsto{\downarrow} x\). Equivalently, \(\widehat D=\Omega S(\Sigma D)\).

\begin{lemma}\label{lem:eta-dcpo-embedding}
For every dcpo \(D\), the map \(\eta_D:D\to\widehat D\) is a dcpo embedding and \(\eta_D[D]\) is a sub-dcpo of \(\widehat D\).
\end{lemma}

\begin{proof}
The map \(\eta_D\) is an order embedding because \({\downarrow} x\subseteq{\downarrow} y\) if and only if \(x\leq y\). If \((x_i)_i\) is directed, then its supremum in \(\widehat D\) is the Scott closure of \(\bigcup_i{\downarrow} x_i\), which equals \({\downarrow}\bigvee_i x_i\). Hence \(\eta_D(\bigvee_i x_i)=\bigvee_i\eta_D(x_i)\), and the image is closed under directed suprema.
\end{proof}

\subsection{The least non-discrete reflective full subcategory}

A map \(u:D\to\two\) is Scott continuous precisely when \(u^{-1}(1)\) is Scott open. Ordered pointwise on the left and by inclusion on the right, this gives a natural isomorphism of dcpos
\(
  [D\to\two]\cong\sigma(D).
\)

Consequently,
\(
 E_{\two}=
 \left\{e:A\to B\mid
 e^{-1}:\sigma(B) \cong \sigma(A)
 \text{ is an order isomorphism}
 \right\}.
\)
Recall that
\(
  \Rtwo=E_{\two}^{\perp}.
\)
Since inverse-image maps preserve arbitrary unions and finite intersections, a bijective inverse-image map is equivalently a frame isomorphism.

\begin{lemma}\label{lem:open-embedding}
Let \(e:A\to B\) be Scott continuous. If
\(
  e^{-1}:\sigma(B)\twoheadrightarrow\sigma(A)
\)
is surjective, then \(e\) is an order embedding and \(e[A]\) is a sub-dcpo of \(B\). Hence \(e:A\to e[A]\) is an isomorphism of dcpos.
\end{lemma}

\begin{proof}
Suppose \(e(x)\le e(y)\). Let \(U\in\sigma(A)\) contain \(x\). Choose \(V\in\sigma(B)\) with \(U=e^{-1}(V)\). Since \(V\) is an upper set,
\(
  e(x)\in V,\  e(x)\le e(y)
  \ \Longrightarrow\  e(y)\in V,
\)
so \(y\in U\). Thus every Scott-open neighbourhood of \(x\) contains \(y\), and the specialization order of the Scott topology gives \(x\le y\). Therefore \(e\) reflects order and is injective. Since \(e\) is Scott continuous and therefore monotone, it is an order embedding.

If \((e(x_i))_i\) is directed in \(e[A]\), then \((x_i)_i\) is directed in \(A\) because \(e\) reflects order. Scott continuity gives
\(
  \bigvee_i e(x_i)=e\left(\bigvee_i x_i\right)\in e[A].
\)
Thus \(e[A]\) is a sub-dcpo and \(e\) is an isomorphism onto its image.
\end{proof}

\begin{theorem}\label{thm:rigidity}
Let \(D\in\Rtwo\). Every morphism \(e:D\to E\) in \(E_{\two}\) is an isomorphism in \(\DCPO\).
\end{theorem}

\begin{proof}
Since \(D\in\Rtwo\), precomposition with \(e\) induces an isomorphism
of dcpos
\(
  e^*:[E\to D]\cong[D\to D].
\)
In particular, the identity map \(\id_D:D\to D\) has a unique
Scott-continuous extension along \(e\). Hence there exists a
Scott-continuous map
\(
  r:E\longrightarrow D
\)
such that
\(
  r\circ e=\id_D.
\)

We show that \(e\circ r=\id_E\). Let \(y\in E\), let
\(W\in\sigma(E)\), and put
\(
  U=e^{-1}(W)\in\sigma(D).
\)
Then
\[
  e^{-1}\bigl(r^{-1}(U)\bigr)
  =(r\circ e)^{-1}(U)
  =U
  =e^{-1}(W).
\]
Since
\(
  e^{-1}:\sigma(E)\longrightarrow\sigma(D)
\)
is injective, it follows that
\(
  r^{-1}(U)=W.
\)
Therefore
\[
  y\in W
  \Longleftrightarrow r(y)\in U
  \Longleftrightarrow e(r(y))\in W.
\]
Thus \(y\) and \(e(r(y))\) have exactly the same Scott-open
neighbourhoods. Since Scott spaces are \(T_0\),
\(
  e(r(y))=y.
\)
Hence
\(
  e\circ r=\id_E,
\)
and \(e\) is an isomorphism.
\end{proof}

\begin{remark}
Surjectivity of \(e^{-1}:\sigma(E)\to\sigma(D)\) already makes \(e\) a dcpo embedding by \cref{lem:open-embedding}. If \(D\in\Rtwo\) and the inverse-image map is an isomorphism, \cref{thm:rigidity} shows that the embedding is onto.
\end{remark}

\begin{corollary}
\label{cor:minimal}
The category \(\Rtwo\) is the least non-discrete reflective full
subcategory of \(\DCPO\).
\end{corollary}

\begin{proof}
Let \(\mathcal K\) be a reflective full subcategory of \(\DCPO\)
containing a non-discrete dcpo \(D\). Reflective full subcategories
are closed under retracts. Hence \cref{lem:retract} implies that
\(\two\in\mathcal K\).
By the defining minimality of the reflective hull,
\(\ReflHull_{\DCPO}(\two)\subseteq\mathcal K\).
Therefore
\(\Rtwo\subseteq\mathcal K\).
\end{proof}

\subsection{Johnstone's dcpo and the Ho--Zhao counterexample}

The Ho--Zhao problem asks whether \(\DCPO\) is $\Gamma$-faithful, namely,
whether two dcpos with isomorphic lattices of Scott-closed sets must be
isomorphic. Ho, Goubault-Larrecq, Jung, and Xi gave a negative answer
in~\cite{ho2018}. Their counterexample is based on a
dcpo $H$ obtained by iterating Johnstone's non-sober dcpo. 

We first recall
Johnstone's dcpo.
Put
\(
  \overline{\mathbb N}=\mathbb N\cup\{\infty\},
\)
where \(\infty\) is greater than every natural number. Johnstone's
dcpo, denoted by \(\mathcal \J\), has underlying set
\[
  \mathcal \J=\mathbb N\times\overline{\mathbb N}.
\]

For \(m,m',n,n'\in\mathbb N\), define
\[
\begin{aligned}
  (m,n)&<_{1}(m,n')
    &&\text{if }n<n',\\
  (m,n)&<_{2}(m,\infty),\\
  (m,n)&<_{3}(m',\infty)
    &&\text{if }n\leq m'.
\end{aligned}
\]
Let
\[
  <\;=\;<_{1}\cup<_{2}\cup<_{3},
  \ 
  \leq\;=\;<\cup=.
\]
Then \(\leq\) is a partial order on \(\mathcal \J\).

\begin{figure}[htb]
	\centering
     %\zihao{-5}\songti
	\includegraphics[width=5cm]{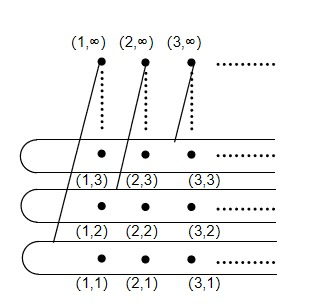}
 \caption{\  Johnstone dcpo \(\J\)}
  \label{fig:图1} % 交叉引用标签（如\ref{fig:vietoris}）
\end{figure}

For \(m\in\mathbb N\), the chain
\[
  C_m=\{m\}\times\overline{\mathbb N}
\]
is called the \(m\)-th component. For
\(n\in\overline{\mathbb N}\), the set
\[
  L_n=\mathbb N\times\{n\}
\]
is called the \(n\)-th level.

The only nontrivial directed subsets of \(\mathcal \J\), up to
cofinal subsets, are the chains
\[
  \{(m,n):n\in N\},
\]
where \(N\subseteq\mathbb N\) is infinite, and
\[
  \bigvee_{n\in N}(m,n)=(m,\infty).
\]
Hence \(\mathcal \J\) is a dcpo.

The Scott-closed irreducible subsets of \(\mathcal \J\) are precisely
the principal ideals and \(\mathcal \J\) itself. Since \(\mathcal \J\)
has no greatest element, it is not sober in its Scott topology.

Next, we recall the construction of $H$ because the same dcpo provides a
natural witness for the properness of $\mathcal{R}_2$.

Let \(\mathbb N^*\) denote the set of all finite strings of natural
numbers. We write
$\varepsilon$
for the empty string, $n.s$ for the string obtained by adjoining \(n\) to the front of \(s\),
and $ts$ for the concatenation of the strings \(t\) and \(s\).

For every \(s\in\mathbb N^*\), let \(\mathcal \J_s\) be a copy of
Johnstone's dcpo, and form the disjoint union
\[
  H'
  =
  \coprod_{s\in\mathbb N^*}\mathcal \J_s.
\]
We write the elements of \(H'\) as triples
\[
  (m,n,s)
  \in
  \mathbb N\times\overline{\mathbb N}\times\mathbb N^*.
\]

Define an equivalence relation \(\sim\) on \(H'\) by
\[
  (m,n,s)\sim(m,\infty,n.s)
\]
for all \(m,n\in\mathbb N\) and \(s\in\mathbb N^*\). Thus the
finite element \((m,n,s)\) of the copy \(\mathcal \J_s\) is
identified with the limit element \((m,\infty,n.s)\) of the copy
\(\mathcal \J_{n.s}\).

\begin{figure}[htb]
	\centering
	\includegraphics[width=6cm]{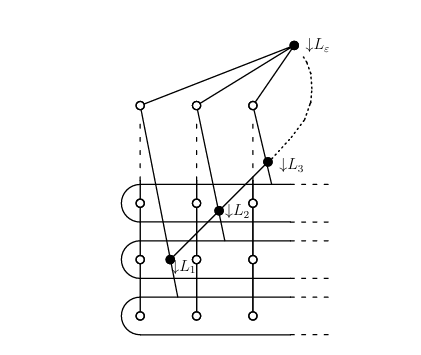}
 \caption{\  The top part of the order sobrification of $H$}
  \label{fig:图3} % 交叉引用标签（如\ref{fig:vietoris}）

\end{figure}

The underlying set of \(H\) is the quotient
\[
  H=H'/\mathord{\sim}.
\]
Every equivalence class contains exactly one element of the form
\[
  (m,\infty,s).
\]
We abbreviate this element to \((m,s)\). Consequently, the
underlying set of \(H\) may be identified with
\[
  \mathbb N\times\mathbb N^*.
\]

Informally, one begins with one copy of Johnstone's dcpo. At each
finite level one attaches a further copy whose limit points are
identified with the points of that level. This procedure is repeated
at every newly created finite level. In the resulting ordered set,
every point is a limit point in some copy of \(\mathcal \J\).

For a nonempty finite string \(t\), let \(\min(t)\) denote its
smallest entry. Define relations on
\(\mathbb N\times\mathbb N^*\) as follows:
\[
\begin{aligned}
  (m,n.s)&<_{1}(m,n'.s)
    &&\text{if }n<n',\\
  (m,ts)&<_{2}(m,s)
    &&\text{if }t\neq\varepsilon,\\
  (m,ts)&<_{3}(m',s)
    &&\text{if }t\neq\varepsilon
      \text{ and }\min(t)\leq m'.
\end{aligned}
\]

For relations \(R,S\), write \(R;S\) for their composite (first \(R\), then \(S\)), and put
\[
  <
  =
  <_{1}\cup<_{2}\cup<_{3}
  \cup(<_{2};<_{1})
  \cup(<_{3};<_{1}).
\]
Finally, define
\[
  \leq\;=\;<\cup=.
\]
\begin{figure}[htb]
	\centering
	\includegraphics[width=6cm]{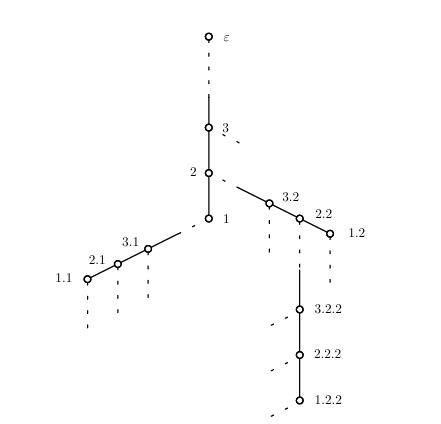}
 \caption{\  The order on $\mathbb{N}^*$}
  \label{fig:图4} % 交叉引用标签（如
\end{figure}
\noindent

\begin{proposition}
\label{prop:H-order}
The relation \(\leq\) is a partial order on
\(\mathbb N\times\mathbb N^*\) by \cite[Proposition~5.1]{ho2018}, and
\(H\) is a dcpo by \cite[Proposition~5.4]{ho2018}.

The construction has the following properties:
\begin{enumerate}[label=\textup{(\roman*)}]
\item the irreducible Scott-closed subsets of \(H\) are exactly the principal ideals \(\down h\) and the layer closures \(\down L_s\)(see Figure 3, s is a string, for example, $L_{n.s}$ is the n-th layer of $\J_s$) \cite[Theorem~5.12]{ho2018};
\item the order-sobrification unit \(\eta_H:H\to \widehat{H}\) induces an isomorphism
\[
  \eta_H^{-1}:\sigma(\widehat{H})\cong\sigma(H);
\]
by \cite[Theorem~4.1 and Section~5.7]{ho2018};
\item \(\widehat{H}\) has a largest element, whereas \(H\) has no largest element.
\end{enumerate}
\end{proposition}

\begin{theorem}\label{thm:proper}
The dcpo \(H\) does not belong to \(\Rtwo\). Consequently, \(\Rtwo\subsetneq\DCPO\).
\end{theorem}

\begin{proof}
By \cref{prop:H-order}, \(\eta_H\in E_{\two}\). If \(H\in\Rtwo\), then \cref{thm:rigidity} would make \(\eta_H\) an isomorphism. This is impossible because \(\widehat H\) has a greatest element and \(H\) does not. Hence \(H\notin\Rtwo\), proving the strict inclusion.
\end{proof}

\begin{remark}
The same pair \((H,\widehat H)\) witnesses the failure of \(\Gamma\)-faithfulness in \(\DCPO\): inverse image along \(\eta_H\) yields \(\Gamma(\widehat H)\cong\Gamma(H)\), but the two dcpos are not isomorphic \cite{ho2018}.
\end{remark}

% ============================================================
% Section 4
% ============================================================

\section{Limits, Skula-closed sub-dcpos, and sober dcpos}
\label{sec:closure}

We next determine several classes of dcpos contained in \(\Rtwo\). A Skula-closed sub-dcpo can be represented as an equalizer of maps into a power of \(\two\); closure under limits then yields the main results of this section.

\subsection{The Skula topology}

Let \(X\) be a \(T_0\) space, with specialization order
\(\leq_X\). For \(x\in X\), put
\(
  {\downarrow} x
  =
  \{y\in X:y\leq_Xx\}.
\)

\begin{definition}[{\cite[V-5.32]{gierz2003}}]
The \emph{Skula topology}, also called the \emph{\(b\)-topology},
on \(X\) is the topology generated by the subbasis
\(
  \mathcal O(X)\cup\{{\downarrow} x:x\in X\}.
\)
A subset that is closed in the Skula topology is called
\emph{Skula closed}, or \emph{\(b\)-closed}.
\end{definition}

Since
\(
  {\downarrow} x=\overline{\{x\}}
\)
is closed in \(X\), the Skula topology is finer than the original
topology.

\begin{lemma}
\label{lem:skula-normal-form}
Let \(X\) be a \(T_0\) space, and let \(A\subseteq X\) be Skula
closed. Then there exist a set \(I\) and open subsets
\(
  V_i\subseteq U_i\subseteq X 
\) for $i \in I$
such that
\(
  A
  =
  \bigcap_{i\in I}
  \bigl(U_i^c\cup V_i\bigr).
\)
\end{lemma}

\begin{proof}
A basic Skula-open subset has the form
\(
  U\cap{\downarrow} x_1\cap\cdots\cap{\downarrow} x_n,
\)
where \(U\) is open in \(X\). Its complement is
\(
  U^c
  \cup
  ({\downarrow} x_1)^c
  \cup\cdots\cup
  ({\downarrow} x_n)^c.
\)
Put
\[
  V
  =
  U\cap
  \bigl(
    ({\downarrow} x_1)^c
    \cup\cdots\cup
    ({\downarrow} x_n)^c
  \bigr).
\]
The set \(V\) is open and satisfies \(V\subseteq U\). Moreover,
\(
\begin{aligned}
  U^c
  \cup
  ({\downarrow} x_1)^c
  \cup\cdots\cup
  ({\downarrow} x_n)^c
  &=
  U^c\cup V.
\end{aligned}
\)
Since every Skula-closed subset is an intersection of complements of
basic Skula-open subsets, the required representation follows.
\end{proof}

\begin{lemma}[{\cite[V-5.32 and V-5.34]{gierz2003}}]\label{lem:skula-dense-sobrification}
Let \(X\) be a \(T_{0}\) space, let \(Y\) be sober, and let
\(i:X\to Y\) be a topological embedding. Then \(i\) is a sobrification
map if and only if \(i(X)\) is Skula dense in \(Y\). Equivalently, this
holds if and only if \(i(X)\cap U\cap{\downarrow} y\neq\varnothing\) for
every open subset \(U\subseteq Y\) and every \(y\in U\).
\end{lemma}

\begin{lemma}[{\cite[Corollary~3.5]{Keimel2009}}]\label{lem:sober-skula}
Let \(Y\) be a sober space, and let \(X\subseteq Y\) have the subspace
topology. Then \(X\) is sober if and only if it is Skula closed in \(Y\).
\end{lemma}

\subsection{Skula-closed sub-dcpos}

\begin{lemma}
\label{lem:skula-equalizer}
Let \(D\) be a dcpo, and let \(A\subseteq D\) be Skula closed in the
Scott space \(\Sigma D\). Then there exist a set \(I\) and
Scott-continuous maps
\(
  f,g:D\longrightarrow\two^I
\)
such that
\(
  A=\Eq(f,g).
\)
In particular, \(A\), equipped with the order inherited from \(D\),
is a dcpo.
\end{lemma}

\begin{proof}
By \cref{lem:skula-normal-form}, there exist Scott-open subsets
\(
  V_i\subseteq U_i\subseteq D
  \) for \( i\in I
\)
such that
\[
  A
  =
  \bigcap_{i\in I}
  \bigl(U_i^c\cup V_i\bigr).
\]

Define maps \(f,g:D\to\two^I\) by
\(
  f(x)(i)=\chi_{V_i}(x),
  \
  g(x)(i)=\chi_{U_i}(x).
\)
Since \(U_i\) and \(V_i\) are Scott open, their characteristic maps into \(\two\) are Scott continuous. Hence \(f\) and \(g\) are Scott
continuous.

Because \(V_i\subseteq U_i\),
\(
  f(x)(i)=g(x)(i)
\)
if and only if
\(
  x\notin U_i
  \ \text{or}\
  x\in V_i.
\)
Consequently,
\[
\begin{aligned}
  \Eq(f,g)
  &=
  \{x\in D:
    f(x)(i)=g(x)(i)
    \text{ for every }i\in I\}\\
  &=
  \bigcap_{i\in I}
  \bigl(U_i^c\cup V_i\bigr)
  =A.
\end{aligned}
\]
Since equalizers in \(\DCPO\) are dcpos with the inherited order,
\(A\) is a dcpo.
\end{proof}

\begin{theorem}
\label{thm:skula-closed}
Let \(D\in\Rtwo\), and let \(A\subseteq D\) be Skula closed in
\(\Sigma D\). Then
\(
  A\in\Rtwo.
\)
\end{theorem}

\begin{proof}
By \cref{lem:skula-equalizer}, there exist Scott-continuous maps
\(
  f,g:D\longrightarrow\two^I
\)
such that
\(
  A=\Eq(f,g).
\)
Moreover,
\(
  \two^I\in\Rtwo.
\)
Since \(D\in\Rtwo\) and \(\Rtwo\) is closed under equalizers,
\(
  A=\Eq(f,g)\in\Rtwo.
\)
\end{proof}

\begin{corollary}
\label{cor:scott-closed}
Let \(D\in\Rtwo\). If \(A\subseteq D\) is Scott closed, then
\(
  A\in\Rtwo.
\)
\end{corollary}

\begin{proof}
Every Scott-closed subset is Skula closed because the Skula topology
is finer than the Scott topology. Apply \cref{thm:skula-closed}.
\end{proof}

\subsection{Specialization orders of sober spaces}

\cref{thm:skula-closed} gives a convenient way to produce new objects of \(\Rtwo\)
from known ones.  We next combine it with the Skula-closed
characterization of sober subspaces to show that specialization
orders of sober spaces also belong to \(\Rtwo\).

\begin{theorem}
\label{thm:sober-specialization}
Let \(\mathcal K\) be a reflective full subcategory of \(\DCPO\)
containing \(\two\). Then
\(
  \Omega X\in\mathcal K
\)
for every sober space \(X\).
In particular,
\(
  \Omega X\in\Rtwo
\)
for every sober space \(X\).
\end{theorem}

\begin{proof}
Let
\(
  \tau=\mathcal O(X).
\)
Consider the evaluation map
\(
  e_X:X\longrightarrow\two^\tau
\)
defined by
\[
  e_X(x)(U)=
  \begin{cases}
    1,&x\in U,\\
    0,&x\notin U,
  \end{cases}
  \  U\in\tau.
\]

Since \(X\) is \(T_0\), its open subsets separate distinct points.
Hence \(e_X\) is injective. Moreover, for every \(U\in\tau\),
\(
  U
  =
  e_X^{-1}
  \bigl(
    \{p\in\two^\tau:p(U)=1\}
  \bigr).
\)
Thus \(e_X\) is a topological embedding.

The pointwise order identifies \(\two^\tau\) with the algebraic complete lattice \(\mathcal P(\tau)\). Its Scott topology has the basis \(\{{\uparrow} F:F\subseteq\tau\text{ finite}\}\), which is exactly the product Sierpi\'nski topology. It is sober by \cite[Theorem~II-1.21]{gierz2003}.

Since \(X\) is sober, its embedded image \(e_X(X)\) is sober. By
\cref{lem:sober-skula}, the subset \(e_X(X)\) is Skula closed in
\(\two^\tau\). Hence, by \cref{lem:skula-equalizer}, there exist a set
\(I\) and Scott-continuous maps
\(
  f,g:\two^\tau\longrightarrow\two^I
\)
such that
\(
  e_X(X)=\Eq(f,g).
\)

Since \(\mathcal K\) contains \(\two\) and is closed under limits,
\(
  \two^\tau\in\mathcal K
  \ \text{and}\
  \two^I\in\mathcal K.
\)
Closure under equalizers now gives
\(
  e_X(X)\in\mathcal K.
\)

It remains to compare the orders. For \(x,y\in X\),
\[
\begin{aligned}
  x\leq_Xy
  &\Longleftrightarrow
  \text{every open set containing \(x\) also contains \(y\)}\\
  &\Longleftrightarrow
  e_X(x)\leq e_X(y).
\end{aligned}
\]
Therefore \(e_X\) induces an order isomorphism
\(
  \Omega X\cong e_X(X).
\)
It follows that
\(
  \Omega X\in\mathcal K.
\)

Taking \(\mathcal K=\Rtwo\) proves the final assertion.
\end{proof}

\begin{corollary}
\label{cor:sober-consequences}
The following dcpos belong to $\Rtwo$.
\begin{enumerate}[label=\textup{(\roman*)}]
\item Every Scott-sober dcpo.
\item Every discrete dcpo.
\item For every dcpo \(D\),
\(
  \widehat D=\Omega S(\Sigma D)\in\Rtwo.
\)
\item If
\(
  \eta_D[D]
  =
  \{{\downarrow} x:x\in D\}
\)
is Skula closed in the standard sobrification \(S(\Sigma D)\), then
\(
  D\in\Rtwo.
\)
\item If \(\eta_D[D]\) is Scott closed in \(\widehat D\), then
\(
  D\in\Rtwo.
\)
\item Every complete lattice.
\end{enumerate}
\end{corollary}

\begin{proof}
For \textup{(i)}, if \(D\) is Scott sober, then \(\Sigma D\) is
sober and
\(
  \Omega(\Sigma D)=D.
\)
The conclusion follows from \cref{thm:sober-specialization}.

Statement \textup{(ii)} follows from \textup{(i)}, since every
discrete space is sober.

For \textup{(iii)}, the standard sobrification \(S(\Sigma D)\) is
sober and its specialization order is
\(
  \Omega S(\Sigma D)=\widehat D.
\)
Apply \cref{thm:sober-specialization}.

For \textup{(iv)}, put \(A=\eta_D[D]\). Since \(A\) is Skula
closed in the sober space \(S(\Sigma D)\), \cref{lem:sober-skula}
implies that \(A\), with the subspace topology, is sober. The map
\(\eta_D:\Sigma D\to A\) is a homeomorphism, so \(\Sigma D\) is sober.
Thus \(D\) is Scott sober, and \textup{(i)} gives \(D\in\Rtwo\).

Statement \textup{(v)} follows from \textup{(iii)} and
\cref{cor:scott-closed}.

For \textup{(vi)}, let \(L\) be a complete lattice and equip it with
its upper topology \(\nu(L)\), generated by the complements of
principal ideals. The space
\(
  (L,\nu(L))
\)
is sober, and its specialization order is the original order of
\(L\) \cite[Chapter~III, Section~3]{gierz2003}. Hence
\(
  L=\Omega(L,\nu(L))\in\Rtwo
\)
by \cref{thm:sober-specialization}.
\end{proof}

\begin{remark}
Part \textup{(iii)} uses only the facts that \(S(\Sigma D)\) is sober
and that
\(
  \widehat D=\Omega S(\Sigma D).
\)
It does not assert that the lower-Vietoris topology of
\(S(\Sigma D)\) coincides with the Scott topology of \(\widehat D\).
\end{remark}

\subsection{Bounded-complete dcpos}

\begin{corollary}
\label{cor:bounded-complete}
Every bounded-complete dcpo belongs to \(\Rtwo\), where bounded-complete means that it has a least element $\bot$ and every nonempty bounded subset has a
supremum.
\end{corollary}

\begin{proof}
Let \(D^\top=D\cup\{\top\}\), where \(x\leq\top\) for every \(x\in D\).
Every subset \(A\subseteq D^\top\) has a supremum: it is \(\bot\) if
\(A=\varnothing\), it is \(\top\) if \(\top\in A\) or if \(A\subseteq D\)
is unbounded in \(D\), and otherwise it is \(\sup_D A\). By the
bounded-completeness convention, \(D\) has a least element \(\bot\).
Hence every nonempty subset of \(D^\top\) has a nonempty set of lower
bounds, whose supremum is its infimum; the infimum of the empty set is
\(\top\). Thus \(D^\top\) is a complete lattice and belongs to \(\Rtwo\)
by \cref{cor:sober-consequences}.

The subset \(D\) is lower in \(D^\top\) and is closed under directed
suprema computed in \(D^\top\). It is therefore Scott closed. By
\cref{cor:scott-closed}, \(D\in\Rtwo\).
\end{proof}
\begin{remark}
The main conclusions of this section may be summarized as follows:
\[
\left\{
\begin{array}{c}
\text{specialization orders of sober spaces},\\
\text{Scott-sober dcpos},\\
\text{discrete dcpos},\\
\text{complete lattices},\\
\text{bounded-complete dcpos},\\
\text{order sobrifications }\widehat D,\\
\text{Skula-closed sub-dcpos of objects of }\Rtwo
\end{array}
\right\}
\subseteq\Rtwo.
\]
In particular,
\(
  \widehat D\in\Rtwo
\)
for every dcpo \(D\).

\end{remark}

\section{\texorpdfstring{$\Gamma$}{Gamma}-faithfulness and maximality}\label{sec:gamma}

We now turn from closure properties to the question of
$\Gamma$-faithfulness.  \cref{lem:Gamma-unit} first shows that the reflection unit \(\lambda_D:D\to \mathbf{L_2}D\) preserves the lattices of Scott-open and Scott-closed subsets.
The remaining problem is therefore to show that two objects of \(\Rtwo\)
with isomorphic Scott-closed-set lattices must themselves be
isomorphic.

\subsection{Representations in the irreducible spectrum}

\begin{lemma}\label{lem:Gamma-unit}
For every dcpo \(D\), the reflection unit \(\lambda_D:D\to \mathbf{L_2}D\) induces isomorphisms
\(
  \lambda_D^{-1}:\sigma(\mathbf{L_2}D)\cong\sigma(D)
  \ \text{and}\
  \Gamma(\mathbf{L_2}D)\cong\Gamma(D).
\)
\end{lemma}

\begin{proof}
By \cref{thm:BKS}, \(\lambda_D\) is \(\two\)-equable. Hence
\(
  \lambda_D^*:[\mathbf{L_2}D\to\two]\cong[D\to\two]
\)
is an isomorphism of dcpos. Under the natural identification of maps into \(\two\) with Scott-open sets, this is exactly the first isomorphism. Taking complements gives the second.
\end{proof}

To compare two dcpos with the same Scott-closed-set lattice, we place
them in a common dcpo. \cref{lem:realization} represents each such dcpo as a
sub-dcpo of \(\Irr(L)\), and \cref{lem:common-generated-realization} produces a common
$L$-representation containing both of them.
\begin{definition}\label{def:realization}
Let \(L\) be a coframe. The nonzero irreducible elements form a dcpo
under the inherited order. Indeed, for a directed family \((p_i)\), put
\(p=\bigvee_i p_i\). Suppose \(p\leq a\vee b\). If neither
\(p\leq a\) nor \(p\leq b\), choose \(i,j\) such that
\(p_i\nleq a\) and \(p_j\nleq b\), and then choose \(k\geq i,j\).
Since \(p_k\leq a\vee b\), irreducibility of \(p_k\) yields
\(p_k\leq a\) or \(p_k\leq b\), a contradiction. Hence \(p\) is
irreducible.

For \(a\in L\), put
\(
  C_a=\{p\in\Irr(L):p\le a\}.
\)
A sub-dcpo \(S\subseteq\Irr(L)\) is called an \emph{\(L\)-representation} if
\[
  \theta_S:L\longrightarrow\Gamma(S),
  \  a\longmapsto S\cap C_a,
\]
is an order isomorphism.
\end{definition}

Each \(C_a\) is Scott closed in \(\Irr(L)\): it is a lower set, and if a
directed family in \(C_a\) has supremum, then that supremum is still below
\(a\). Consequently \(S\cap C_a\in\Gamma(S)\) for every sub-dcpo \(S\).

\begin{lemma}\label{lem:realization}
If \(\Gamma(D)\cong L\), then \(D\) embeds isomorphically into \(\Irr(L)\) as an \(L\)-representation.
\end{lemma}

\begin{proof}
Let \(\alpha:\Gamma(D)\to L\) be a lattice isomorphism, and define
\(
  j_D:D\longrightarrow\Irr(L),
  \  x\longmapsto\alpha(\down x).
\)
Principal ideals are irreducible, and
\(
 x\le y
 \Longleftrightarrow \down x\subseteq\down y
 \Longleftrightarrow j_D(x)\le j_D(y),
\)
so \(j_D\) is an order embedding. Moreover,
\(
  \bigvee_i\down x_i=\down\bigvee_i x_i
\)
for every directed family, hence the image is a sub-dcpo. Finally, for
\(x\in D\), \(j_D(x)\in C_a\) iff \(\alpha(\down x)\le a\), which is
equivalent to \(\down x\subseteq\alpha^{-1}(a)\). Thus
\(j_D[D]\cap C_a\) is precisely the image of \(\alpha^{-1}(a)\) under
\(j_D\), and \(j_D[D]\) is an \(L\)-representation.
\end{proof}

\begin{lemma}
\label{lem:common-generated-realization}
Let \(D,E\subseteq P=\Irr(L)\) be two \(L\)-representations, and let
\(
U=\langle D\cup E\rangle_{\mathrm{dcpo}}
\)
be the smallest sub-dcpo of \(P\) containing \(D\cup E\). Then \(U\) is also
an \(L\)-representation. Consequently, the inclusion maps

\[
  i_D:D\hookrightarrow U,
  \ 
  i_E:E\hookrightarrow U
\]
belong to \(E_{\two}\).
\end{lemma}

\begin{proof}
We first record a fact about generated sub-dcpos. Let \(P\) be a dcpo,
let \(B\subseteq P\), and let
\(
U=\langle B\rangle_{\mathrm{dcpo}}.
\)
If \(C\subseteq P\) is Scott closed, then
\[
U\cap C=\langle B\cap C\rangle_{\mathrm{dcpo}}.
\]

Indeed, put
\(
H=\langle B\cap C\rangle_{\mathrm{dcpo}}.
\)
Clearly \(H\subseteq U\cap C\). Define
\(
K=H\cup(U\setminus C).
\)
We claim that \(K\) is a sub-dcpo of \(U\). Let \(A\subseteq K\) be
directed. If \(A\subseteq H\), then \(\bigvee A\in H\). Otherwise, \(A\)
contains some \(a\notin C\). Since \(a\leq\bigvee A\) and \(C\) is a
lower set, \(\bigvee A\notin C\). Hence \(\bigvee A\in U\setminus C\).
Thus \(K\) is a sub-dcpo of \(U\).

Moreover, \(K\) contains \(B\). By the minimality of \(U\), we have
\(U\subseteq K\). It follows that
\(
U\cap C\subseteq H,
\)
and hence
\(
U\cap C=\langle B\cap C\rangle_{\mathrm{dcpo}}.
\)

We shall also use the following related fact: if \(u\in U\setminus C\),
then there exists \(b\in B\setminus C\) such that \(b\leq u\). Define
\(
K'=
\left\{
u\in U:
u\in C
\text{ or there exists }b\in B\setminus C\text{ such that }b\leq u
\right\}.
\)
The set \(K'\) contains \(B\). It is also a sub-dcpo of \(U\). Indeed,
let \(A\subseteq K'\) be directed. If \(\bigvee A\in C\), then
\(\bigvee A\in K'\). Otherwise, since \(C\) is Scott closed, some
\(a\in A\) lies outside \(C\). By the definition of \(K'\), there exists
\(b\in B\setminus C\) such that
\(
b\leq a\leq\bigvee A.
\)
Thus \(\bigvee A\in K'\). By the minimality of \(U\), we obtain \(K'=U\).

Now take
\(
B=D\cup E
\)
and define
\(
\theta_U(a)=U\cap C_a,
\ 
  C_a=\{p\in\Irr(L):p\le a\}.
\)
The map
\(
\theta_U:L\longrightarrow\Gamma(U)
\)
is injective, since its restriction to \(D\) is the representation map
\(
\theta_D(a)=D\cap C_a,
\)
which is injective.

To prove surjectivity, let \(F\in\Gamma(U)\). Since \(D\) and \(E\) are
\(L\)-representations, there exist unique \(a,b\in L\) such that
\(
F\cap D=D\cap C_a, F\cap E=E\cap C_b.
\)

We first prove that \(a=b\). Let \(d\in D\cap C_a\). Then \(d\in F\).
If \(e\in E\) and \(e\leq d\), the lowerness of \(F\) gives
\(
e\in F\cap E=E\cap C_b.
\)
Therefore
\(
E\cap C_d\subseteq E\cap C_b.
\)
Since \(E\) is an \(L\)-representation, it follows that \(d\leq b\). Hence
\(
D\cap C_a\subseteq D\cap C_b.
\)
Since \(D\) is an \(L\)-representation, we obtain \(a\leq b\). By symmetry,
\(b\leq a\), and therefore \(a=b\).

Consequently,
\(
F\cap B=B\cap C_a.
\)
By the first generating fact,
\(
U\cap C_a
=
\langle B\cap C_a\rangle_{\mathrm{dcpo}}
\subseteq F.
\)

Conversely, suppose that \(u\in F\setminus C_a\). By the second generating
fact, there exists \(b_0\in B\setminus C_a\) such that
\(
b_0\leq u.
\)
Since \(F\) is a lower set, \(b_0\in F\). Thus
\(
b_0\in F\cap B=B\cap C_a,
\)
contradicting \(b_0\notin C_a\). Therefore
\(
F\subseteq U\cap C_a.
\)
It follows that
\(
F=U\cap C_a.
\)
Hence \(\theta_U\) is surjective. It is also order-reflecting: if
\(\theta_U(a)\subseteq\theta_U(b)\), intersecting with \(D\) gives
\(\theta_D(a)\subseteq\theta_D(b)\), and the order-isomorphism \(\theta_D\)
implies \(a\le b\). Hence \(\theta_U\) is an order isomorphism.
Thus \(U\) is an \(L\)-representation.

Finally, let \(i_D:D\hookrightarrow U\) be the inclusion map. The following
diagram commutes:
\[
\begin{tikzcd}[column sep=large,row sep=large]
L
  \arrow[r,"\theta_U","\cong"']
  \arrow[d,"\mathrm{id}_L"']
&
\Gamma(U)
  \arrow[d,"i_D^{-1}", "\cong"']
\\
L
  \arrow[r,"\theta_D"']
&
\Gamma(D).
\end{tikzcd}
\]
Indeed, for every \(a\in L\),
\(
i_D^{-1}\bigl(\theta_U(a)\bigr)
=(U\cap C_a)\cap D
=D\cap C_a
=\theta_D(a).
\)
Thus \(i_D^{-1}:\Gamma(U)\to\Gamma(D)\) is an isomorphism. Taking
complements shows that the inverse-image map
\(
i_D^{-1}:\sigma(U)\longrightarrow\sigma(D)
\)
is also an isomorphism. Hence \(i_D\in E_{\two}\).

The same argument, using the analogous commutative diagram
\[
\begin{tikzcd}[column sep=large,row sep=large]
L
  \arrow[r,"\theta_U","\cong"']
  \arrow[d,"\mathrm{id}_L"']
&
\Gamma(U)
  \arrow[d,"i_E^{-1}", "\cong"']
\\
L
  \arrow[r,"\theta_E"']
&
\Gamma(E),
\end{tikzcd}
\]
shows that \(i_E:E\hookrightarrow U\) also belongs to \(E_{\two}\).
\end{proof}

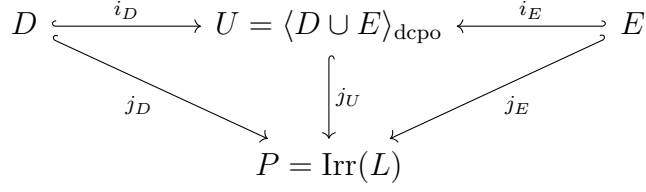
\begin{figure}[ht]
\centering
\begin{tikzcd}[column sep=huge,row sep=large]
D \arrow[r,hook,"i_D"] \arrow[dr,hook,"j_D"']
  & U=\langle D\cup E\rangle_{\mathrm{dcpo}}
      \arrow[d,hook,"j_U"]
  & E \arrow[l,hook',"i_E"'] \arrow[dl,hook',"j_E"] \\
  & P=\Irr(L) &
\end{tikzcd}
\caption{The common generated representation inside \(P=\Irr(L)\).}
\label{fig:common-realization}
\end{figure}

\subsection{Faithfulness and maximality}

\begin{theorem}\label{thm:gamma-faithful}
The category \(\Rtwo\) is \(\Gamma\)-faithful on objects: if \(D,E\in\Rtwo\) and \(\Gamma(D)\cong\Gamma(E)\), then \(D\cong E\).
\end{theorem}

\begin{proof}
Let \(L\) be a common copy of the two closed-set lattices. By \cref{lem:realization}, we may regard \(D\) and \(E\) as two \(L\)-representations inside \(\Irr(L)\). Let
\(
  U=\langle D\cup E\rangle_{\mathrm{dcpo}}.
\)
By the preceding lemma, both inclusions
\(
  i_D:D\hookrightarrow U,
  i_E:E\hookrightarrow U
\)
belong to \(E_{\two}\). Since \(D,E\in E_{\two}^{\perp}\), there are maps
\(
  r_D:U\to D,
  r_E:U\to E
\)
with \(r_D\circ i_D=\id_D\) and \(r_E\circ i_E=\id_E\). By \cref{thm:rigidity}, both inclusions are isomorphisms. Hence \(D\cong U\cong E\).
\end{proof}

\cref{thm:gamma-faithful} proves $\Gamma$-faithfulness of \(\Rtwo\). \cref{lem:Gamma-unit} also
shows that no proper full supercategory remains \(\Gamma\)-faithful: if a dcpo \(D\notin\Rtwo\) is added,
then $D$ and $\mathbf{L_2}D$ have isomorphic Scott-closed-set lattices but are
not isomorphic.

\begin{theorem}\label{thm:maximal}
Let \(\mathcal A\) be a full subcategory of \(\DCPO\). If
\(
  \Rtwo\subsetneq\mathcal A,
\)
then \(\mathcal A\) is not \(\Gamma\)-faithful. Consequently, \(\Rtwo\) is inclusion-maximal among \(\Gamma\)-faithful full subcategories of \(\DCPO\).
\end{theorem}

\begin{proof}
Choose \(D\in\mathcal A\setminus\Rtwo\). Since \(\mathbf{L_2}D\in\Rtwo\subseteq\mathcal A\), \cref{lem:Gamma-unit} gives
\(
  \Gamma(D)\cong\Gamma(\mathbf{L_2}D).
\)
If \(D\cong \mathbf{L_2}D\), then \(D\) would be isomorphic to an object of \(\Rtwo\), hence would itself belong to \(\Rtwo\), a contradiction. Thus \(\mathcal A\) contains two non-isomorphic objects with isomorphic Scott-closed-set lattices.
\end{proof}

\section{Weakly dominated dcpos and the Xu--Zhao example}\label{sec:weakdom}

We compare \(\Rtwo\) with the dominated dcpos of Ho--Goubault-Larrecq--Jung--Xi \cite{ho2018} and the weakly dominated dcpos of Miao--Hou--Jia--Li \cite{miao2024}. We prove \(\WDomi\subseteq\Rtwo\) and use the stratified Johnstone dcpo \(\LJ\) of Xu and Zhao \cite[Example~4.5]{XuZhao2019} to show that the inclusion is strict.

\subsection{\texorpdfstring{Domination and \(C\)-compactness}{Domination and C-compactness}}

For a dcpo \(P\), \(\Irr(\Gamma(P))\) denotes the dcpo of nonempty irreducible Scott-closed subsets ordered by inclusion. For every directed family \(\mathcal D\subseteq\Irr(\Gamma(P))\), its supremum is \(\operatorname{cl}_{\sigma(P)}(\bigcup\mathcal D)\) \cite[Remark~2.1(1)]{XuZhao2019}.

\begin{definition}[{\cite[Definition~2.5]{ho2018}}]\label{def:lhd-irr}
For \(A,B\in\Irr(\Gamma(P))\), write \(A\lhd B\) if \(A\subseteq{\downarrow} b\) for some \(b\in B\), and put \(\nabla_{\mathrm{irr}}B=\{A\in\Irr(\Gamma(P)):A\lhd B\}\).
\end{definition}

\begin{definition}[{\cite[Definition~2.10]{ho2018}}]\label{def:dominated}
A dcpo \(P\) is \emph{dominated} if \(\nabla_{\mathrm{irr}}B\) is Scott closed in \(\Irr(\Gamma(P))\) for every \(B\in\Irr(\Gamma(P))\). The class of dominated dcpos is denoted by \(\DOMI\).
\end{definition}

\begin{definition}[{\cite[Definitions~3.1 and~4.2]{HoZhao2009}; \cite[Definition~5.1]{miao2024}}]\label{def:Ccompact}
Let \(Q\) be a poset. For \(x,y\in Q\), write \(x\prec^*y\) if, for every nonempty Scott-closed subset \(\mathcal A\subseteq Q\) whose supremum exists, \(y\leq\bigvee\mathcal A\) implies \(x\in\mathcal A\). An element \(x\) is \emph{\(C\)-compact} if \(x\prec^*x\). For a dcpo \(P\), let \(C(\Gamma(P))\) be the set of \(C\)-compact elements of \(\Gamma(P)\).
\end{definition}

\begin{proposition}[{\cite[Proposition~4.3(iii) and Corollary~4.6]{HoZhao2009}; \cite[Proposition~5.2 and Corollary~5.3]{miao2024}}]\label{prop:Ccompact-basic}
For every dcpo \(P\), the set \(C(\Gamma(P))\) is a sub-dcpo of \(\Gamma(P)\), and \({\downarrow} p\in C(\Gamma(P))\) for every \(p\in P\).
\end{proposition}

\begin{lemma}\label{lem:Ccompact-irreducible}
For every dcpo \(P\), \(C(\Gamma(P))\setminus\{\varnothing\}\subseteq\Irr(\Gamma(P))\). The empty set is also \(C\)-compact.
\end{lemma}

\begin{proof}
Every nonempty Scott-closed subset of \(\Gamma(P)\) contains its least element \(\varnothing\), so \(\varnothing\prec^*\varnothing\). Let \(B\in C(\Gamma(P))\) be nonempty and suppose \(B\subseteq F_1\cup F_2\), where \(F_1,F_2\in\Gamma(P)\). The set \({\downarrow} F_1\cup{\downarrow} F_2\) is a nonempty Scott-closed subset of \(\Gamma(P)\) with supremum \(F_1\cup F_2\). Since \(B\) is \(C\)-compact, \(B\subseteq F_1\) or \(B\subseteq F_2\). Hence \(B\) is irreducible.
\end{proof}

\begin{definition}[\cite{miao2024}]
\label{def:Ccompact-prelim}
The following notions will be used throughout.
\begin{enumerate}[label=(\roman*)]

\item \label{def:Ccompact-subset}
{\cite[Definition~5.5]{miao2024}} 
A subset \(A\) of a poset \(Q\) is a \emph{\(C\)-compact subset} if
\(\operatorname{cl}_{\sigma(Q)}(A)\) is a \(C\)-compact element of \(\Gamma(Q)\).

\item \label{def:Ccompact-complete}
{\cite[Definition~5.11]{miao2024}}
A dcpo \(Q\) is \emph{\(C\)-compactly complete} if every \(C\)-compact subset of \(Q\) has a supremum.

\item \label{def:prec}
{\cite[Definition~5.14]{miao2024}}
Let \(Q\) be a \(C\)-compactly complete dcpo. For \(x,y\in Q\), write
\(x\prec y\) if, for every Scott-closed \(C\)-compact subset \(A\subseteq Q\),
\(y\leq\bigvee A\) implies \(x\in A\). An element \(x\) is
\emph{\(\prec\)-compact} if \(x\prec x\).

\item \label{def:lhd-C}
{\cite[Definition~5.9]{miao2024}}
For \(A,B\in C(\Gamma(P))\), write \(A\lhd B\) if
\(A\subseteq{\downarrow} b\) for some \(b\in B\), and put
\(\nabla_C B=\{A\in C(\Gamma(P)):A\lhd B\}\).

\item \label{def:weakly-dominated}
{\cite[Definition~5.13]{miao2024}}
A dcpo \(P\) is \emph{weakly dominated} if \(\nabla_C B\) is Scott closed in
\(C(\Gamma(P))\) for every \(B\in C(\Gamma(P))\). The class of weakly dominated dcpos is denoted by \(\WDomi\).

\end{enumerate}
\end{definition}

\begin{proposition}[{\cite[Proposition~5.12]{miao2024}}]\label{prop:Cgamma-Ccomplete}
For every dcpo \(P\), the dcpo \(C(\Gamma(P))\) is \(C\)-compactly complete.
\end{proposition}

\begin{proposition}[{\cite[Propositions~5.15 and~5.17]{miao2024}}]\label{prop:weakdom-characterization}
For \(A,B\in C(\Gamma(P))\), \(A\lhd B\) implies \(A\prec B\). If \(P\) is weakly dominated, then the \(\prec\)-compact elements of \(C(\Gamma(P))\) are exactly the principal ideals \({\downarrow} p\), \(p\in P\).
\end{proposition}

\begin{proposition}[{\cite[Section~6 and Example~6.4]{miao2024}}]\label{prop:DOMI-in-WDomi}
Every dominated dcpo is weakly dominated, and the inclusion is strict: \(\DOMI\subsetneq\WDomi\).
\end{proposition}

\begin{proof}
Let \(P\) be dominated and \(B\in C(\Gamma(P))\). If \(B=\varnothing\), then \(\nabla_C B=\varnothing\), which is Scott closed. Suppose that \(B\neq\varnothing\). By \cref{lem:Ccompact-irreducible}, \(B\in\Irr(\Gamma(P))\), and
\(
  \nabla_C B=\{\varnothing\}\cup\bigl(\nabla_{\mathrm{irr}}B\cap C(\Gamma(P))\bigr).
\)
The set \(\nabla_C B\) is lower. Let \(\mathcal D\subseteq\nabla_C B\) be directed. If \(\mathcal D=\{\varnothing\}\), its supremum belongs to \(\nabla_C B\). Otherwise \(\mathcal D'=\mathcal D\setminus\{\varnothing\}\) is a directed subset of \(\nabla_{\mathrm{irr}}B\), and \(\bigvee\mathcal D'=\bigvee\mathcal D\). Since \(P\) is dominated, this supremum belongs to \(\nabla_{\mathrm{irr}}B\); since \(C(\Gamma(P))\) is a sub-dcpo, it also belongs to \(C(\Gamma(P))\). Hence \(\bigvee\mathcal D\in\nabla_C B\). Strictness is witnessed by \cite[Example~6.4]{miao2024}.
\end{proof}

We next give two sufficient conditions for membership in \(\Rtwo\). We begin with a simple invariance observation.

\begin{lemma}
\label{lem:Ccompact-order-invariant}
Let \(f:Q\to Q'\) be an order isomorphism between dcpos.
Then:
\begin{enumerate}[label=\textnormal{(\roman*)},leftmargin=2.8em]
\item
\(f\) preserves Scott-closed subsets and all existing suprema;
\item
for \(x,y\in Q\),
\(
    x\prec^{*}y
    \ \Longleftrightarrow\ 
    f(x)\prec^{*}f(y);
\)
in particular, \(f\) preserves \(C\)-compact elements;
\item
if \(Q\) and \(Q'\) are \(C\)-compactly complete, then
\(
    x\prec y
    \ \Longleftrightarrow\ 
    f(x)\prec f(y).
\)
\end{enumerate}
\end{lemma}

\begin{proof}
An order isomorphism between dcpos preserves directed suprema and
is therefore a Scott homeomorphism.  Hence it preserves
Scott-closed subsets.  Assertions (ii) and (iii) follow directly
from \cref{def:Ccompact,def:Ccompact-subset,def:prec}.
\end{proof}

% --------------------------------------------------------------------
\subsection{Weakly dominated dcpos}
% --------------------------------------------------------------------

\begin{theorem}
\label{thm:WDomi-in-R2}
Every weakly dominated dcpo belongs to \(\Rtwo\).  Consequently,
\(
    \WDomi\subseteq\Rtwo.
\)
\end{theorem}

\begin{proof}
Let \(P\) be weakly dominated, put \(Q=\mathbf{L_2}P\), and write \(\lambda=\lambda_P:P\to Q\). By \cref{lem:Gamma-unit}, inverse image induces an order isomorphism \(\Phi=\lambda^{-1}:\Gamma(Q)\cong\Gamma(P)\). By \cref{lem:Ccompact-order-invariant}, it restricts to an order isomorphism \(\Phi_C:C(\Gamma(Q))\cong C(\Gamma(P))\) preserving \(\prec\).

Fix \(q\in Q\). By \cref{prop:Ccompact-basic}, the principal ideal \({\downarrow} q\) is \(C\)-compact. Since \({\downarrow} q\lhd{\downarrow} q\), \cref{prop:weakdom-characterization} gives \({\downarrow} q\prec{\downarrow} q\). Hence \(\Phi({\downarrow} q)\) is \(\prec\)-compact in \(C(\Gamma(P))\). The same proposition yields a unique element \(r(q)\in P\) such that
\begin{equation}\label{eq:r-def}
  \Phi({\downarrow} q)={\downarrow} r(q).
\end{equation}
This defines a monotone map \(r:Q\to P\), because \(q\leq q'\) implies \({\downarrow} r(q)=\Phi({\downarrow} q)\subseteq\Phi({\downarrow} q')={\downarrow} r(q')\).

Let \(A\subseteq Q\) be directed. Since \(\Phi\) is a complete-lattice isomorphism,
\(
{\downarrow} r\!\left(\bigvee A\right)
 =\Phi\!\left({\downarrow}\bigvee A\right)
 =\Phi\!\left(\bigvee_{q\in A}{\downarrow} q\right)
 =\bigvee_{q\in A}{\downarrow} r(q)
 ={\downarrow}\bigvee_{q\in A}r(q).
\)
Thus \(r(\bigvee A)=\bigvee_{q\in A}r(q)\), so \(r\) is Scott continuous.

By \cref{lem:Gamma-unit,lem:open-embedding}, \(\lambda\) is a dcpo embedding. Consequently, for every \(p\in P\), \(\Phi({\downarrow}\lambda(p))=\lambda^{-1}({\downarrow}\lambda(p))={\downarrow} p\). Equation~\eqref{eq:r-def} gives \(r\lambda=\id_P\). For \(q\in Q\),
\(
  \Phi({\downarrow}\lambda(r(q)))
  ={\downarrow} r(q)
  =\Phi({\downarrow} q).
\)
Injectivity of \(\Phi\) gives \({\downarrow}\lambda(r(q))={\downarrow} q\), hence \(\lambda r=\id_Q\). Thus \(P\cong \mathbf{L_2}P\in\Rtwo\), and therefore \(P\in\Rtwo\).
\end{proof}

% --------------------------------------------------------------------
\subsection{\texorpdfstring{\(\Gamma\)-unique dcpos}{Csigma-unique dcpos}}

\begin{definition}[{\cite[p.~1]{ZhaoXu2018}}]\label{def:Csigma-unique}
A dcpo \(P\) is \emph{\(\Gamma\)-unique} if \(\Gamma(P)\cong\Gamma(Q)\) implies \(P\cong Q\) for every dcpo \(Q\).
\end{definition}

\begin{proposition}\label{prop:unique-in-R2}
Every \(\Gamma\)-unique dcpo belongs to \(\Rtwo\).
\end{proposition}

\begin{proof}
If \(P\) is \(\Gamma\)-unique, \cref{lem:Gamma-unit} gives \(\Gamma(\mathbf{L_2}P)\cong\Gamma(P)\), so \(P\cong \mathbf{L_2}P\). Since \(\mathbf{L_2}P\in\Rtwo\) and \(\Rtwo\) is closed under isomorphisms, \(P\in\Rtwo\).
\end{proof}

\subsection{The Xu--Zhao dcpo}
% ====================================================================

We now recall the stratified Johnstone dcpo introduced by Xu and Zhao
\cite[Example~4.5]{XuZhao2019}.  Let
\(
\J=\N\times(\N\cup\{\infty\})
\)
be Johnstone's dcpo, ordered by
\(
(i,j)\leq_\J(k,\ell)
\)
if and only if
\(
(i=k\ \text{and}\ j\leq \ell) \ or \ (\ell=\infty\ \text{and}\ j\leq k).
\)

The underlying set of the stratified dcpo is
\(
\LJ=
\{(n,m,h):n,h\in\N,\;
               m\in\N\cup\{\infty\}\}
\cup
\{(n,\infty,\infty):n\in\N\}.
\)
For \(h\in\N\), write
\(
L_h=
\{(n,m,h):n\in\N,\;
               m\in\N\cup\{\infty\}\}.
\)

The order printed in \cite[Example~4.5]{XuZhao2019} is the following.
For points \((n,m,h)\) and \((s,t,k)\) of \(\LJ\),
\(
(n,m,h)\leq (s,t,k)
\)
if and only if one of the following conditions holds:
\begin{enumerate}
\item[(J1)] \(h=k\) and \((n,m)\leq_\J(s,t)\);
\item[(J2)] \(h<k\), \(t=\infty\), and \((n,m)\leq_\J(s,t)\);
\item[(J3)] \(h<k=t=\infty\), \(n\leq s\), and
  \((n,m)\leq_\J(s,\infty)\).
\end{enumerate}

Thus every finite stratum \(L_h\) is a copy of the Johnstone dcpo,
while the additional clauses describe the order between different
strata and the points \((n,\infty,\infty)\).

For \(h\in\N\), put
\(
A_h={\downarrow} L_h, A_\infty=\LJ.
\)
The structural properties of \(\LJ\) that will be used below are the
following:

\begin{lemma}[\cite{XuZhao2019}]
\label{lem:XZ-properties}
For the stratified Johnstone dcpo \(\LJ\) of
\cite[Example~4.5]{XuZhao2019}, the following properties hold.

\begin{enumerate}[label=\textnormal{(\arabic*)},leftmargin=3.2em]

\item
\label{xz:layers}
For every \(h\in\N\),
\begin{equation}
\label{eq:Ah-layers}
  A_h
  =
  \bigcup_{1\leq t\leq h}L_t,
\end{equation}
and \(A_h\) is Scott closed
\cite[Example~4.5(3),(4)]{XuZhao2019}.

\item
\label{xz:global-sup}
For every \(n\in\N\),
\begin{equation}
\label{eq:global-sup}
  (n,\infty,\infty)
  =
  \bigvee \{(n,\infty,h) : h\in\N\}.
\end{equation}
Consequently,
\[
  % \bigvee_{h\in\N}^{\Gm(\LJ)}A_h
  \bigvee\nolimits_{\Gm(\LJ)} \{A_h: h\in\N\}
  =
  A_\infty
  =
  \LJ.
\]
The supremum relation \eqref{eq:global-sup} is used in the proof of
\cite[Theorem~4.11]{XuZhao2019}; together with
\cite[Example~4.5(3),(4)]{XuZhao2019}, it gives the displayed
supremum in \(\Gm(\LJ)\).

\item
\label{xz:irreducibles}
The irreducible Scott-closed subsets of \(\LJ\) are precisely
\[
  \Irr(\Gamma(\LJ))
  =
  \{{\downarrow} x:x\in\LJ\}
  \cup
  \{A_h:h\in\N\}
  \cup
  \{A_\infty\}
\]
\cite[Theorem~4.9]{XuZhao2019}.

\item
\label{xz:principal-bound}
For every \(h\in\N\),
\[
  A_h
  \subseteq
  {\downarrow}(h+1,\infty,\infty).
\]
This is the principal-bound relation appearing in the proof of
\cite[Proposition~4.10]{XuZhao2019}.

\item
\label{xz:no-greatest}
The points
  $(n,\infty,\infty), n\in\N$,
  are maximal elements of \(\LJ\); in particular, \(\LJ\) has no
greatest element.  See the construction in
\cite[Example~4.5]{XuZhao2019}.

\item
\label{xz:unique}
The dcpo \(\LJ\) is \(\Gamma\)-unique
\cite[Theorem~4.11]{XuZhao2019}.

\end{enumerate}
\end{lemma}

\begin{corollary}
\label{cor:LJ-in-R2}
The dcpo \(\LJ\) belongs to \(\Rtwo\).
\end{corollary}

\begin{proof}
By \cref{lem:XZ-properties}\ref{xz:unique},
\(\LJ\) is \(\Gamma\)-unique.  Apply
\cref{prop:unique-in-R2}.
\end{proof}

% ====================================================================
\subsection{\texorpdfstring{\(C\)-compact Scott-closed sets}
{C-compact Scott-closed sets}}
% ====================================================================

We next establish two results about \(C\)-compact Scott-closed sets that are needed below.

\begin{lemma}[{\cite[Proposition~4.1]{HoZhao2009}}]\label{lem:union-closed}
Let \(P\) be a dcpo and \(\mathcal A\in\Gamma(\Gamma(P))\). Then \(\bigcup\mathcal A\) is Scott closed in \(P\), and \(\bigvee_{\Gamma(P)}\mathcal A=\bigcup\mathcal A\).
\end{lemma}

For \(h,r\in\N\), put
\begin{equation}
\label{eq:Khr}
  K_{h,r}
  =
  \{(i,k,t):
      1\leq t\leq h,\;
      i\in\N,\;
      1\leq k\leq r\}.
\end{equation}

\begin{lemma}
\label{lem:Khr}
For every \(h,r\in\N\):
\begin{enumerate}[label=\textnormal{(\roman*)},leftmargin=2.8em]
\item
  \(K_{h,r}\) is Scott closed in \(\LJ\);
\item
  \(K_{h,r}\subseteq{\downarrow}(r,\infty,h);\)

\item
  \((K_{h,r})_{r\in\N}\) is increasing in \(\Gamma(\LJ)\), and
  \(
  \bigvee\nolimits_{\Gm(\LJ)} \{K_{h,r}: r\in\N\}=A_h
  \)
\end{enumerate}
\end{lemma}

\begin{proof}
Fix \(h,r\in\N\).  A point below a finite point
\((i,k,t)\) must lie in the same finite layer and the same
Johnstone column, with finite second coordinate at most \(k\).
Consequently, \(K_{h,r}\) is lower.

Let \(D\subseteq K_{h,r}\) be directed and choose \(d_0\in D\).
For every \(d\in D\), choose \(e\in D\) with
\(d_0,d\leq e\).
Since \(e\) is finite, the finite-layer order forces
\(d_0,d,e\) to have the same layer and the same first coordinate.
Thus \(D\) is contained in a finite chain
\(
  \{(i,k,t):1\leq k\leq r\}
\)
for fixed \(i,t\).  It therefore has a largest element, which is its
supremum.  Hence \(K_{h,r}\) is Scott closed.

If \((i,k,t)\in K_{h,r}\), then \(t\leq h\) and \(k\leq r\). By the definition of $\J$, $(i,k)\leq_{\J}(r,\infty)$. The finite-layer order therefore gives $(i,k,t)\leq(r,\infty,h)$, and (ii) follows.

Finally, the family \((K_{h,r})_{r\in\N}\) is increasing.  Its union
is the set \(F_h\) of all finite points in the layers
\(L_t\), \(1\leq t\leq h\).  By
\cref{lem:XZ-properties}\ref{xz:layers},
\(A_h\) is Scott closed and contains \(F_h\), hence
\(
  \cl(F_h)\subseteq A_h.
\)

Conversely, for every \(t\leq h\) and \(i\in\N\),
\(
  (i,\infty,t)
  =
  \bigvee_{k\in\N}(i,k,t).
\)
Thus every nonfinite point of \(L_t\) belongs to \(\cl(F_h)\).
Using
\(
  A_h=\bigcup_{1\leq t\leq h}L_t,
\)
we obtain
\(
  A_h\subseteq\cl(F_h).
\)
Hence
\(
  \cl(F_h)=A_h,
\)
which is precisely the supremum of
\((K_{h,r})_{r\in\N}\) in \(\Gamma(\LJ)\).
\end{proof}

\begin{proposition}
\label{prop:Ah-Ccompact}
For every \(h\in\N\),
\(
  A_h\in C(\Gamma(\LJ)).
\)
\end{proposition}

\begin{proof}
Fix \(h\in\N\).  Let
\(\mathcal A\in\Gamma(\Gamma(\LJ))\)
be nonempty and suppose
\(
  A_h
  \subseteq
  \bigvee_{\Gamma(\LJ)}\mathcal A.
\)
By \cref{lem:union-closed},
\(
  \bigvee_{\Gamma(\LJ)}\mathcal A
  =
  \bigcup\mathcal A.
\)
For every \(r\in\N\),
\((r,\infty,h)\in A_h\),
so there exists \(B_r\in\mathcal A\) such that
\((r,\infty,h)\in B_r\).
Since \(B_r\) is lower, ${\downarrow}(r,\infty,h)\subseteq B_r$. Since \(\mathcal A\) is lower in \(\Gamma(\LJ)\), ${\downarrow}(r,\infty,h)\in\mathcal A$.

By \cref{lem:Khr}(ii), $K_{h,r}\in\mathcal A$.
The family \((K_{h,r})_{r\in\N}\) is directed and has supremum
\(A_h\) by \cref{lem:Khr}(iii).  Since \(\mathcal A\) is Scott
closed, $A_h\in\mathcal A$.

Thus
\(
  A_h\prec^{*}A_h,
\)
and therefore
\(
  A_h\in C(\Gamma(\LJ)).
\)
\end{proof}

\begin{proposition}
\label{prop:Ainfty-Ccompact}
The Scott-closed set
\(
  A_\infty=\LJ
\)
belongs to \(C(\Gamma(\LJ))\).
\end{proposition}

\begin{proof}
By \cref{lem:XZ-properties}\ref{xz:global-sup},
\(
  \bigvee\nolimits_{\Gm(\LJ)} \{A_h: h\in\N\}
  =
  A_\infty.
\)
Let
\(\mathcal A\in\Gamma(\Gamma(\LJ))\)
be nonempty and suppose
\(
  A_\infty
  \subseteq
  \bigvee_{\Gamma(\LJ)}\mathcal A.
\)
By \cref{lem:union-closed},
\(
  \bigvee_{\Gamma(\LJ)}\mathcal A
  =
  \bigcup\mathcal A.
\)
Hence
\(
  A_h\subseteq\bigcup\mathcal A
  \ (h\in\N).
\)
By \cref{prop:Ah-Ccompact},
\(
  A_h\in\mathcal A
  \ (h\in\N).
\)

Since \(\mathcal A\) is Scott closed and
\(
\bigvee\nolimits_{\Gm(\LJ)} \{A_h: h\in\N\}=A_\infty,
\)
we conclude that
\(
  A_\infty\in\mathcal A.
\)
Thus
\(
  A_\infty\prec^{*}A_\infty,
\)
and therefore
\(
  A_\infty\in C(\Gamma(\LJ)).
\)
\end{proof}

% ====================================================================
\subsection{The Xu--Zhao dcpo is not weakly dominated}
% ====================================================================

We can now obtain the desired separation directly from the
Xu--Zhao properties recorded in
\cref{lem:XZ-properties}.

\begin{theorem}
\label{thm:LJ-not-WDomi}
The dcpo \(\LJ\) is not weakly dominated.
\end{theorem}

\begin{proof}
By
Proposition ~\ref{prop:Ah-Ccompact}, Proposition ~\ref{prop:Ainfty-Ccompact},
\(
  A_h,A_\infty\in C(\Gamma(\LJ))
  \ (h\in\N).
\)

For every \(h\in\N\),
\cref{lem:XZ-properties}\ref{xz:principal-bound} gives
\(
  A_h
  \subseteq
  {\downarrow}(h+1,\infty,\infty).
\)
Since
\(
  (h+1,\infty,\infty)\in A_\infty,
\)
we have
\(
  A_h\lhd A_\infty.
\)
Hence
\(
  A_h\in\nabla_C A_\infty
  \ (h\in\N).
\)

By \cref{lem:XZ-properties}\ref{xz:global-sup},
the family \((A_h)_{h\in\N}\) is directed in \(\Gamma(\LJ)\) and
\(
  \bigvee\nolimits_{\Gm(\LJ)} \{A_h: h\in\N\}
  =
  A_\infty.
\)
Since \(A_\infty\in C(\Gamma(\LJ))\), the same element is the supremum
of \((A_h)_{h\in\N}\) in the sub-dcpo \(C(\Gamma(\LJ))\).

On the other hand,
\(
  A_\infty\ntriangleleft A_\infty.
\)
Indeed, if \(A_\infty\lhd A_\infty\), there would exist
\(x\in A_\infty=\LJ\) such that
\(
  \LJ
  =
  A_\infty
  \subseteq
  {\downarrow} x.
\)
Then \(x\) would be a greatest element of \(\LJ\), contradicting
\cref{lem:XZ-properties}\ref{xz:no-greatest}.

Thus \(\nabla_C A_\infty\) contains the directed family
\(
  (A_h)_{h\in\N}
\)
but does not contain its supremum \(A_\infty\).  Therefore
\(\nabla_C A_\infty\) is not Scott closed in \(C(\Gamma(\LJ))\), and
\(\LJ\) is not weakly dominated.
\end{proof}

\begin{theorem}
\label{thm:LJ-separation}
The dcpo \(\LJ\) satisfies
\(
    \LJ\in\Rtwo\setminus\WDomi.
\)
Consequently,
\(
  \WDomi\subsetneq\Rtwo.
\)
\end{theorem}

\begin{proof}
By \cref{cor:LJ-in-R2}, \(\LJ\in\Rtwo\), whereas
\cref{thm:LJ-not-WDomi} gives \(\LJ\notin\WDomi\). Hence
\(\LJ\in\Rtwo\setminus\WDomi\). Together with
\(\WDomi\subseteq\Rtwo\) from \cref{thm:WDomi-in-R2}, this yields
\(\WDomi\subsetneq\Rtwo\).
\end{proof}

% ====================================================================

\subsection{Consequences}
% ====================================================================

The preceding results give a precise comparison between dominated dcpos, weakly dominated dcpos, and \(\two\)-complete dcpos.

\begin{theorem}
\label{thm:class-hierarchy}
The classes satisfy
\(
    \DOMI
    \subsetneq
    \WDomi
    \subsetneq
    \Rtwo
    \subsetneq
    \DCPO.
\)
\end{theorem}

\begin{proof}
By \cref{prop:DOMI-in-WDomi}, \(\DOMI\subsetneq\WDomi\).

By \cref{thm:WDomi-in-R2},
\(
  \WDomi\subseteq\Rtwo,
\)
and \cref{thm:LJ-separation} provides
\(
  \LJ\in\Rtwo\setminus\WDomi.
\)
Hence
\(
  \WDomi\subsetneq\Rtwo.
\)

Finally,
\(
  \Rtwo\subsetneq\DCPO
\)
by \cref{thm:proper}.  Combining the three strict inclusions yields
\(
  \DOMI
  \subsetneq
  \WDomi
  \subsetneq
  \Rtwo
  \subsetneq
  \DCPO.
\)
\end{proof}

% ============================================================
% Section 7
% ============================================================

\section{An internal construction of the reflection}
\label{sec:internal-reflection}

We now represent the reflector inside the order sobrification. Write \(\widehat D=\Irr(\Gamma(D))\) and, using \cref{lem:eta-dcpo-embedding}, identify \(D\) with the principal copy \(\eta_D[D]\subseteq\widehat D\), where \(\eta_D(x)={\downarrow} x\).
Thus points of $D$ will be regarded as principal irreducible Scott-closed
sets.  Recall that if $(p_i)_{i\in I}$ is directed in $\widehat D$, then its
supremum is computed in $\Gamma(D)$ as
\begin{equation}
\label{eq:directed-sup-Dsharp}
\bigvee_{i\in I}p_i
=
\operatorname{cl}_{\sigma(D)}
\!\left(\bigcup_{i\in I}p_i\right).
\end{equation}
By \cref{cor:sober-consequences}(iii), $\widehat D\in\Rtwo$ for every dcpo $D$.

The construction below has two steps. First, every \(\two\)-equable extension of $D$ is represented canonically as a sub-dcpo of $\widehat D$.
Second, among the sub-dcpos between $D$ and $\widehat D$, the extensions that
preserve the Scott-open-set lattice are characterized by an internal trace
condition.  Iterating the resulting support test produces a descending
transfinite chain whose stable value is the image of the reflection
$\mathbf{L_2}D$ in $\widehat D$.

\subsection{Internal images of \texorpdfstring{\(\two\)-equable morphisms}{2-equable morphisms}}

We begin with the density property that makes the internal representation
possible.

\begin{lemma}
\label{lem:skula-density-equable}
Let $e:D\to E$ belong to $E_{\two}$.  Then, for every $y\in E$ and every
$W\in\sigma(E)$ with $y\in W$,
\(
W\cap \mathord{{\downarrow}}_E y\cap e[D]\neq\varnothing.
\)
\end{lemma}

\begin{proof}
Suppose, to the contrary, that
\(
W\cap \mathord{{\downarrow}}_E y\cap e[D]=\varnothing.
\)
Since $\mathord{{\downarrow}}_E y$ is Scott closed,
$W\setminus\mathord{{\downarrow}}_E y$ is Scott open.  Moreover,
\(
e^{-1}(W)
=
e^{-1}\!\left(W\setminus\mathord{{\downarrow}}_E y\right).
\)
This contradicts the injectivity of
\(
e^{-1}:\sigma(E)\longrightarrow\sigma(D),
\)
because $y\in W$ implies
$W\neq W\setminus\mathord{{\downarrow}}_E y$.
\end{proof}

\begin{lemma}
\label{lem:internal-representation}
Let $e:D\to E$ belong to $E_{\two}$.  For $y\in E$, define
\(
\rho_e(y)=\{x\in D:e(x)\leq y\}.
\)
Then the following statements hold.
\begin{enumerate}
\item $\rho_e(y)$ is a nonempty irreducible Scott-closed subset of $D$;
\item $\rho_e:E\to \widehat D$ is a dcpo embedding;
\item $\rho_e(e(x))=\mathord{{\downarrow}}x$ for every $x\in D$.
\end{enumerate}
Consequently, every \(\two\)-equable extension of $D$ is isomorphic,
over $D$, to a sub-dcpo of $\widehat D$ containing the canonical image of
$D$.
\end{lemma}

\begin{proof}
Fix $y\in E$.  By \cref{lem:skula-density-equable}, applied to
the Scott-open set $E$, there exists $d\in D$ with $e(d)\leq y$; hence
$\rho_e(y)$ is nonempty.

The set $\rho_e(y)$ is a lower set.  If $(x_i)_{i\in I}$ is directed in
$\rho_e(y)$, then Scott continuity of $e$ gives
\(
e\!\left(\bigvee_{i\in I}x_i\right)
=
\bigvee_{i\in I}e(x_i)
\leq y,
\)
so $\bigvee_i x_i\in\rho_e(y)$.  Thus $\rho_e(y)$ is Scott closed.

To prove irreducibility, let $U,V\in\sigma(D)$ and assume
\(
\rho_e(y)\cap U\neq\varnothing, \rho_e(y)\cap V\neq\varnothing.
\)
Since $e^{-1}:\sigma(E)\to\sigma(D)$ is surjective, choose
$U',V'\in\sigma(E)$ such that
\(
e^{-1}(U')=U, e^{-1}(V')=V.
\)
If $x\in\rho_e(y)\cap U$, then $e(x)\leq y$ and $e(x)\in U'$.  As $U'$ is
an upper set, $y\in U'$.  Similarly, $y\in V'$.  Hence
$y\in U'\cap V'$.  \cref{lem:skula-density-equable} yields
$d\in D$ such that
\(
e(d)\leq y, e(d)\in U'\cap V'.
\)
Therefore
\(
d\in\rho_e(y)\cap U\cap V,
\)
which proves that $\rho_e(y)$ is irreducible.

The map $\rho_e$ is monotone.  To see that it reflects order, suppose
$\rho_e(y)\subseteq\rho_e(z)$ and $y\nleq z$.  Since the specialization order
of the Scott topology is the original order, there exists $W\in\sigma(E)$
with $y\in W$ and $z\notin W$.  By
\cref{lem:skula-density-equable}, choose $d\in D$ such that
$e(d)\leq y$ and $e(d)\in W$.  Then
$d\in\rho_e(y)\subseteq\rho_e(z)$, so $e(d)\leq z$.  Since $W$ is an upper
set, this forces $z\in W$, a contradiction.  Thus $y\leq z$, and $\rho_e$
is an order embedding.

It remains to verify preservation of directed suprema.  Let
$(y_i)_{i\in I}$ be directed in $E$, and put $y=\bigvee_i y_i$.  Since each
$\rho_e(y_i)\subseteq\rho_e(y)$ and $\rho_e(y)$ is Scott closed,
\begin{equation}
\label{eq:rho-one-inclusion}
\operatorname{cl}_{\sigma(D)}
\!\left(\bigcup_i\rho_e(y_i)\right)
\subseteq \rho_e(y).
\end{equation}
For the converse, let $x\in\rho_e(y)$ and let $U\in\sigma(D)$ contain $x$.
Choose $U'\in\sigma(E)$ with $e^{-1}(U')=U$.  Since $e(x)\leq y$ and
$e(x)\in U'$, the upperness of $U'$ gives $y\in U'$.  Scott openness then
yields an index $i$ such that $y_i\in U'$.  Applying
\cref{lem:skula-density-equable} to $y_i$ and $U'$, choose
$d\in D$ with
\(
e(d)\leq y_i,
\ 
e(d)\in U'.
\)
Thus $d\in\rho_e(y_i)\cap U$.  Every Scott-open neighbourhood of $x$
therefore meets $\bigcup_i\rho_e(y_i)$, and hence
\(
x\in
\operatorname{cl}_{\sigma(D)}
\!\left(\bigcup_i\rho_e(y_i)\right).
\)
Together with \eqref{eq:rho-one-inclusion} and
\eqref{eq:directed-sup-Dsharp}, this proves
\(
\rho_e\!\left(\bigvee_i y_i\right)
=
\bigvee_i\rho_e(y_i).
\)
Hence $\rho_e$ is a dcpo embedding.

Finally, \cref{lem:open-embedding} implies that $e$ reflects order.  Therefore, for every
$x\in D$,
\(
\rho_e(e(x))
=
\{d\in D:e(d)\leq e(x)\}
=
\mathord{{\downarrow}}x.
\)
The final assertion follows because the image of a dcpo embedding is a
sub-dcpo and $\rho_e\circ e=\eta_D$.
\end{proof}

\subsection{The internal trace criterion}

Set $P=\widehat D$ and continue to identify $D$ with $\eta_D[D]\subseteq P$.
If
\(
D\subseteq S\subseteq P
\)
is a sub-dcpo and $p\in S$, define the $D$-trace of $p$ in $S$ by
\begin{equation}
\label{eq:D-trace}
\operatorname{Tr}^{S}_{D}(p)
=
D\cap\mathord{{\downarrow}}_{S}p.
\end{equation}

\begin{definition}
\label{def:trace-support}
A point $p\in S$ is \emph{supported by its $D$-trace in $S$} if
\(
p\in
\operatorname{cl}_{\sigma(S)}
\!\left(\operatorname{Tr}^{S}_{D}(p)\right).
\)
Equivalently, for every $C\in\Gamma(S)$,
\begin{equation}
\label{eq:closed-support-test}
\operatorname{Tr}^{S}_{D}(p)\subseteq C
\ \Longrightarrow\ 
p\in C.
\end{equation}
\end{definition}

The next lemma supplies the surjectivity part of the open-set comparison.

\begin{lemma}
\label{lem:diamond-open-extension}
For $U\in\sigma(D)$, define
\(
\Diamond U
=
\{p\in \widehat D:p\cap U\neq\varnothing\}.
\)
Then $\Diamond U\in\sigma(\widehat D)$.  Consequently, for every sub-dcpo
$D\subseteq S\subseteq \widehat D$,
\(
U^{\sharp}_{S}:=S\cap\Diamond U
\)
is Scott open in $S$ and satisfies
\(
U^{\sharp}_{S}\cap D=U,
\)
under the canonical identification of $D$ with its principal points.
\end{lemma}

\begin{proof}
The set $\Diamond U$ is an upper set.  Let $(p_i)_{i\in I}$ be directed in
$\widehat D$, and let
\(
p=\bigvee_i p_i
=
\operatorname{cl}_{\sigma(D)}
\!\left(\bigcup_i p_i\right)
\)
belong to $\Diamond U$.  Then $U$ meets
$\operatorname{cl}_{\sigma(D)}(\bigcup_i p_i)$.  Since $U$ is open, it
already meets $\bigcup_i p_i$.  Hence $p_i\cap U\neq\varnothing$ for some
$i$, so $p_i\in\Diamond U$.  Thus $\Diamond U$ is Scott open.

Now let $\mathord{{\downarrow}}x$ be a principal point.  Because $U$ is an
upper set,
\(
\mathord{{\downarrow}}x\cap U\neq\varnothing
\ \Longleftrightarrow\ 
x\in U.
\)
The final assertion follows.
\end{proof}

\begin{theorem}
\label{thm:internal-trace-criterion}
Let $D\subseteq S\subseteq \widehat D$ be a sub-dcpo, and let
$i:D\hookrightarrow S$ be the inclusion.  Then
\(
i\in E_{\two}
\)
if and only if every point of $S$ is supported by its $D$-trace in $S$.
\end{theorem}

\begin{proof}
Assume first that
\(
i^{-1}:\sigma(S)\cong\sigma(D)
\)
is an isomorphism.  Fix $p\in S$ and $W\in\sigma(S)$ with $p\in W$.  If
\(
W\cap\operatorname{Tr}^{S}_{D}(p)=\varnothing,
\)
then $W\setminus\mathord{{\downarrow}}_{S}p$ is Scott open.  Moreover,
\(
\left(W\setminus\mathord{{\downarrow}}_{S}p\right)\cap D
=
W\cap D.
\)
Indeed, a point of $W\cap D$ lying below $p$ would belong to
$W\cap\operatorname{Tr}^{S}_{D}(p)$.  Since
$W\setminus\mathord{{\downarrow}}_{S}p\neq W$, the displayed equality
contradicts injectivity of $i^{-1}$.  Hence every Scott-open neighbourhood
of $p$ meets $\operatorname{Tr}^{S}_{D}(p)$, so $p$ is supported.

Conversely, assume every point of $S$ is supported.  By
\cref{lem:diamond-open-extension}, the restriction map
$i^{-1}:\sigma(S)\to\sigma(D)$ is surjective.  To prove injectivity, let
$W_1,W_2\in\sigma(S)$ satisfy
\(
W_1\cap D=W_2\cap D.
\)
Take $p\in W_1$.  Since $p$ is supported, there exists
\(
d\in W_1\cap\operatorname{Tr}^{S}_{D}(p).
\)
Then $d\in W_2$ and $d\leq p$.  As $W_2$ is an upper set, $p\in W_2$.
Thus $W_1\subseteq W_2$, and the reverse inclusion follows symmetrically.
Hence $W_1=W_2$.  Thus $i^{-1}$ is bijective; since inverse-image maps
preserve arbitrary unions and finite intersections, it is a frame
isomorphism, hence an isomorphism of dcpos.
\end{proof}

\subsection{Closed traces and the support operator}

The open-set criterion has a useful dual formulation in terms of Scott
closed sets.  It also connects the present construction directly with the
representations in the irreducible spectrum used in \cref{sec:gamma}.

For $F\in\Gamma(D)$, define its standard closed extension to $\widehat D$ by
\begin{equation}
\label{eq:closed-extension-hat}
\widehat F
=
\{p\in \widehat D:p\subseteq F\}.
\end{equation}
Since $F$ is Scott closed, $\widehat F$ is Scott closed in $\widehat D$.
Indeed, it is a lower set, and if a directed family $(p_i)$ is contained in
$\widehat F$, then \eqref{eq:directed-sup-Dsharp} and the closedness of $F$
show that $\bigvee_i p_i\subseteq F$.

For a sub-dcpo $D\subseteq S\subseteq \widehat D$, put
\begin{equation}
\label{eq:ThetaS}
\Theta_S:\Gamma(D)\longrightarrow\Gamma(S),
\ 
\Theta_S(F)=S\cap\widehat F.
\end{equation}
For $C\in\Gamma(S)$, define its restriction to the original dcpo by
\begin{equation}
\label{eq:closed-restriction}
C|_D
:=
\{x\in D:\mathord{{\downarrow}}x\in C\}
=
\eta_D^{-1}(C\cap\eta_D[D])\in\Gamma(D).
\end{equation}
We say that $S$ satisfies the \emph{closed-trace condition} if
\begin{equation}
\label{eq:closed-trace-condition}
C=S\cap\widehat{C|_D}
\ (C\in\Gamma(S)).
\end{equation}

\begin{proposition}
\label{prop:closed-trace-criterion}
Let $D\subseteq S\subseteq \widehat D$ be a sub-dcpo.  The following are
equivalent.
\begin{enumerate}
\item The inclusion $D\hookrightarrow S$ belongs to $E_{\two}$.
\item The dcpo $S$ satisfies the closed-trace condition.
\item The map
\(
\Theta_S:\Gamma(D)\cong\Gamma(S)
\)
is an order isomorphism.
\end{enumerate}
\end{proposition}

\begin{proof}
For every $C\in\Gamma(S)$ one always has
\begin{equation}
\label{eq:C-contained-in-hat-trace}
C\subseteq S\cap\widehat{C|_D}.
\end{equation}
Indeed, if $p\in C$ and $x\in p$, then the principal point
$\mathord{{\downarrow}}x$ satisfies
$\mathord{{\downarrow}}x\leq p$.  Since $C$ is a lower set,
$\mathord{{\downarrow}}x\in C$, hence $x\in C|_D$.  Therefore
$p\subseteq C|_D$, and so $p\in S\cap\widehat{C|_D}$.

Assume first that every point of $S$ is supported by its $D$-trace.  If
$p\in S\cap\widehat{C|_D}$, then
\(
\operatorname{Tr}^{S}_{D}(p)\subseteq C.
\)
Since $C$ is Scott closed, the support condition gives $p\in C$.  Therefore
\(
C=S\cap\widehat{C|_D}
\ (C\in\Gamma(S)),
\)
so the closed-trace condition holds.

Conversely, suppose that the closed-trace condition holds. By
\eqref{eq:C-contained-in-hat-trace}, this means
\begin{equation}
\label{eq:no-splitting-equality}
C=S\cap\widehat{C|_D}
\ (C\in\Gamma(S)).
\end{equation}
Let $p\in S$ and let $C\in\Gamma(S)$ contain
$\operatorname{Tr}^{S}_{D}(p)$.  If $x\in p$, then
$\mathord{{\downarrow}}x\in\operatorname{Tr}^{S}_{D}(p)\subseteq C$.
Thus $p\subseteq C|_D$, so
$p\in S\cap\widehat{C|_D}=C$.  By
\eqref{eq:closed-support-test}, $p$ is supported.  The equivalence of (1)
and (2) now follows from \cref{thm:internal-trace-criterion}.

It remains to compare (2) and (3). For every $F\in\Gamma(D)$,
\(
\eta_D^{-1}\!\left(\Theta_S(F)\cap\eta_D[D]\right)=F.
\)
Indeed, $\mathord{{\downarrow}}x\in\widehat F$ if and only if
$\mathord{{\downarrow}}x\subseteq F$, which is equivalent to $x\in F$.
Consequently, $\Theta_S$ is injective and reflects order: if
$\Theta_S(F)\subseteq\Theta_S(G)$, intersection with $\eta_D[D]$ and
application of $\eta_D^{-1}$ give $F\subseteq G$. By
\eqref{eq:no-splitting-equality}, the closed-trace condition is exactly the
assertion that every $C\in\Gamma(S)$ has the form $C=\Theta_S(C|_D)$.
Thus $\Theta_S$ is surjective if and only if the closed-trace condition
holds, and in that case it is an order isomorphism.
\end{proof}

\begin{remark}
\label{rem:realization-map}
For $L=\Gamma(D)$, the map $\Theta_S$ in \eqref{eq:ThetaS} is precisely the
representation map of \cref{def:realization}, because
\(
\widehat F
=
\{p\in\Irr(L):p\leq F\}.
\)
Thus \cref{prop:closed-trace-criterion} identifies the
\(\two\)-equable sub-dcpos of $\widehat D$ with the
$\Gamma(D)$-representations considered in \cref{sec:gamma}.
\end{remark}

\begin{definition}
\label{def:support-operator}
For every sub-dcpo $D\subseteq S\subseteq \widehat D$, define
\begin{equation}
\label{eq:support-operator}
\Supp_D(S)
=
\left\{
 p\in S:
 p\in
 \operatorname{cl}_{\sigma(S)}
 \!\left(\operatorname{Tr}^{S}_{D}(p)\right)
\right\}.
\end{equation}
Equivalently,
\begin{equation}
\label{eq:support-closed-form}
p\in\Supp_D(S)
\ \Longleftrightarrow\ 
\forall C\in\Gamma(S),\ 
\operatorname{Tr}^{S}_{D}(p)\subseteq C
\Longrightarrow p\in C.
\end{equation}
\end{definition}

\begin{theorem}
\label{thm:support-subdcpo}
For every sub-dcpo $D\subseteq S\subseteq \widehat D$, the set
$\Supp_D(S)$ is a sub-dcpo of $S$ containing $D$.
\end{theorem}

\begin{proof}
If $p=\mathord{{\downarrow}}x\in D$, then
$p\in\operatorname{Tr}^{S}_{D}(p)$, so $p\in\Supp_D(S)$.  Hence
$D\subseteq\Supp_D(S)$.

Let $(p_i)_{i\in I}$ be directed in $\Supp_D(S)$ and put
$p=\bigvee_i p_i$, where the supremum is computed in $S$.  Let
$C\in\Gamma(S)$ satisfy
\(
\operatorname{Tr}^{S}_{D}(p)\subseteq C.
\)
Since $p_i\leq p$,
\(
\operatorname{Tr}^{S}_{D}(p_i)
\subseteq
\operatorname{Tr}^{S}_{D}(p)
\subseteq C.
\)
By \eqref{eq:support-closed-form}, $p_i\in C$ for every $i$.  Since $C$ is
Scott closed, $p=\bigvee_i p_i\in C$.  Another application of
\eqref{eq:support-closed-form} gives $p\in\Supp_D(S)$.  Therefore
$\Supp_D(S)$ is a sub-dcpo of $S$.
\end{proof}

\begin{corollary}
\label{cor:fixed-point-characterization}
For a sub-dcpo $D\subseteq S\subseteq \widehat D$, the following are
equivalent:
\[
S=\Supp_D(S)\ \Longleftrightarrow \
D\hookrightarrow S\in E_{\two} \ \Longleftrightarrow \ 
\Theta_S:\Gamma(D)\cong\Gamma(S).
\]
\end{corollary}

\begin{proof}
The equality $S=\Supp_D(S)$ says exactly that every point of $S$ is
supported by its $D$-trace.  Apply \cref{thm:internal-trace-criterion}
and \cref{prop:closed-trace-criterion}.
\end{proof}

We shall call a sub-dcpo $D\subseteq S\subseteq \widehat D$
\emph{$D$-admissible} if it satisfies the equivalent conditions of
\cref{cor:fixed-point-characterization}.

\subsection{Transfinite iteration}

Starting from the full order sobrification, define a transfinite decreasing
sequence $(S_\alpha)$ by
\begin{equation}
\label{eq:support-chain}
S_0=\widehat D,
\quad
S_{\alpha+1}=\Supp_D(S_\alpha),
\quad
S_\lambda=\bigcap_{\alpha<\lambda}S_\alpha
\ (\lambda\text{ limit}).
\end{equation}
By \cref{thm:support-subdcpo} and transfinite induction, every
$S_\alpha$ is a sub-dcpo of $\widehat D$ containing $D$.  At a limit stage,
the intersection is again a sub-dcpo because directed suprema are computed
in the ambient dcpo $\widehat D$.

Let $\kappa=|\widehat D|$. The sequence stabilizes before $\kappa^+$.
Indeed, if $S_\alpha\neq S_{\alpha+1}$ for every $\alpha<\kappa^+$,
choose $p_\alpha\in S_\alpha\setminus S_{\alpha+1}$. Since the sequence is
decreasing, the points $p_\alpha$ are pairwise distinct, giving an injection
$\kappa^+\to \widehat D$, a contradiction. Hence
$S_\beta=S_{\beta+1}$ for some $\beta<\kappa^+$.
We denote the stable value by
\begin{equation}
\label{eq:S-infty}
S_\infty(D):=S_\beta.
\end{equation}
The value is independent of the particular stabilizing ordinal, since the
sequence is constant from the first fixed point onward.

\begin{theorem}
\label{thm:maximal-admissible}
The dcpo $S_\infty(D)$ is the largest $D$-admissible sub-dcpo of
$\widehat D$.  Equivalently, it is the largest sub-dcpo
\(
D\subseteq T\subseteq \widehat D
\)
for which the inclusion $D\hookrightarrow T$ belongs to $E_{\two}$.
\end{theorem}

\begin{proof}
By stability,
\(
S_\infty(D)=\Supp_D(S_\infty(D)).
\)
Hence \cref{cor:fixed-point-characterization} shows that
$S_\infty(D)$ is $D$-admissible.

Now let $T$ be any $D$-admissible sub-dcpo of $\widehat D$.  We prove by
transfinite induction that
\(
T\subseteq S_\alpha
, \text{for every ordinal }\alpha.
\)
The initial step is immediate.  Suppose $T\subseteq S_\alpha$, and fix
$p\in T$.  Let $C\in\Gamma(S_\alpha)$ satisfy
\(
\operatorname{Tr}^{S_\alpha}_{D}(p)\subseteq C.
\)
Because $D\subseteq T\subseteq S_\alpha$ and all orders are inherited from
$\widehat D$,
\(
\operatorname{Tr}^{T}_{D}(p)
=
\operatorname{Tr}^{S_\alpha}_{D}(p).
\)
The set $C\cap T$ is Scott closed in $T$, and it contains
$\operatorname{Tr}^{T}_{D}(p)$.  Since $T$ is $D$-admissible,
\cref{thm:internal-trace-criterion} implies that $p$ is supported by
its $D$-trace in $T$.  Therefore
\(
p\in C\cap T\subseteq C.
\)
By \eqref{eq:support-closed-form}, $p\in\Supp_D(S_\alpha)=S_{\alpha+1}$.
Thus $T\subseteq S_{\alpha+1}$.  At a limit ordinal, the conclusion follows
by intersection.  Hence $T\subseteq S_\infty(D)$.
\end{proof}

\subsection{Identification with the reflection}

Let
\(
\lambda_D:D\longrightarrow \mathbf{L_2}D
\)
be the $\Rtwo$-reflection unit. By \cref{lem:Gamma-unit},
$\lambda_D\in E_{\two}$.  \cref{lem:internal-representation} therefore
gives a dcpo embedding
\(
\rho_{\lambda_D}:\mathbf{L_2}D\longrightarrow \widehat D.
\)
Put
\begin{equation}
\label{eq:T-lambda}
T_\lambda:=\rho_{\lambda_D}[\mathbf{L_2}D]\subseteq \widehat D.
\end{equation}

\begin{lemma}
\label{lem:T-lambda-admissible}
The sub-dcpo $T_\lambda$ is $D$-admissible.
\end{lemma}

\begin{proof}
The map
\(
\rho_{\lambda_D}:\mathbf{L_2}D\longrightarrow T_\lambda
\)
is a dcpo isomorphism, and
\cref{lem:internal-representation} gives
\(
\rho_{\lambda_D}(\lambda_D(x))=\mathord{{\downarrow}}x
\ (x\in D).
\)
Thus the inclusion $D\hookrightarrow T_\lambda$ is isomorphic, as a map
under $D$, to $\lambda_D$.  Since $\lambda_D\in E_{\two}$, the inclusion
$D\hookrightarrow T_\lambda$ also belongs to $E_{\two}$.
\end{proof}

The next lemma is the universality step needed to identify the maximal
admissible sub-dcpo with the reflection image.

\begin{lemma}
\label{lem:admissible-contained-Tlambda}
Let $D\subseteq T\subseteq \widehat D$ be $D$-admissible.  Then
\(
T\subseteq T_\lambda.
\)
\end{lemma}

\begin{proof}
Let $i:D\hookrightarrow T$ denote the inclusion.  Since $T$ is
$D$-admissible, $i\in E_{\two}$.  Because
$\mathbf{L_2}D\in\Rtwo=E_{\two}^{\perp}$, the morphism $i$ is
$\mathbf{L_2}D$-equable.  Hence there exists a unique Scott-continuous map
\(
f:T\longrightarrow \mathbf{L_2}D
\)
such that
\begin{equation}
\label{eq:f-extends-lambda}
f\circ i=\lambda_D.
\end{equation}

We claim that $f\in E_{\two}$.  On Scott-open sets,
\eqref{eq:f-extends-lambda} gives
\(
i^{-1}\circ f^{-1}=\lambda_D^{-1}.
\)
Both
\(
i^{-1}:\sigma(T)\cong\sigma(D)
\ \text{and}\ 
\lambda_D^{-1}:\sigma(\mathbf{L_2}D)\cong\sigma(D)
\)
are isomorphisms.  Therefore
\begin{equation}
\label{eq:f-open-isomorphism}
f^{-1}
=
(i^{-1})^{-1}\circ\lambda_D^{-1}
:
\sigma(\mathbf{L_2}D)\cong\sigma(T),
\end{equation}
so $f\in E_{\two}$. In particular, \cref{lem:open-embedding} implies that $f$ reflects
order.

Let $t\in T$.  Using \eqref{eq:f-extends-lambda} and order reflection of
$f$, we obtain
\begin{align*}
\rho_{\lambda_D}(f(t))
&=
\{x\in D:\lambda_D(x)\leq f(t)\}\\
&=
\{x\in D:f(i(x))\leq f(t)\}\\
&=
\{x\in D:i(x)\leq t\}\\
&=
\{x\in D:\mathord{{\downarrow}}x\subseteq t\}\\
&=t.
\end{align*}
The last equality uses that $t$, as a point of $\widehat D$, is a lower
subset of $D$.  Hence $t\in T_\lambda$, proving
$T\subseteq T_\lambda$.
\end{proof}

\begin{theorem}
\label{thm:internal-formula-reflection}
For every dcpo $D$, the embedding associated with the reflection unit
restricts to a canonical dcpo isomorphism
\(
\rho_{\lambda_D}:\mathbf{L_2}D \cong S_\infty(D).
\)
Equivalently, $\mathbf{L_2}D$ is represented inside
\(\widehat D=\Irr(\Gamma(D))\) as the largest sub-dcpo
containing $D$ for which restriction induces an isomorphism of Scott-open
set lattices.
\end{theorem}

\begin{proof}
By \cref{lem:T-lambda-admissible}, $T_\lambda$ is $D$-admissible.
\cref{thm:maximal-admissible} therefore gives
\(
T_\lambda\subseteq S_\infty(D).
\)
Conversely, $S_\infty(D)$ is $D$-admissible, so
\cref{lem:admissible-contained-Tlambda} gives
\(
S_\infty(D)\subseteq T_\lambda.
\)
Hence
\(
S_\infty(D)=T_\lambda.
\)
Since $\rho_{\lambda_D}:\mathbf{L_2}D\to T_\lambda$ is a dcpo isomorphism, the
result follows.
\end{proof}

\subsection{Consequences}

\begin{corollary}
\label{cor:R2-fixed-pruning}
For every dcpo $D$,
\[
D\in\mathcal R_2
\ \Longleftrightarrow\ 
S_\infty(D)=D,
\]
where the equality on the right is taken inside $\widehat D$ after the
canonical identification $x\leftrightarrow\mathord{{\downarrow}}x$.
\end{corollary}

\begin{proof}
If $D\in\Rtwo$, then the reflection unit $\lambda_D$ is an
isomorphism.  Hence its internal image is exactly the principal copy of
$D$, so \cref{thm:internal-formula-reflection} gives
$S_\infty(D)=D$.

Conversely, if $S_\infty(D)=D$, then
\cref{thm:internal-formula-reflection} gives $\mathbf{L_2}D\cong D$.  Since
$\mathbf{L_2}D\in\Rtwo$ and $\Rtwo$ is closed under isomorphisms, it follows that
$D\in\Rtwo$.
\end{proof}

\begin{remark}
\label{rem:no-monotonicity-needed}
The defining condition of the support operator is evaluated in the current ambient sub-dcpo $S$,
so the Scott-closed sets used in \eqref{eq:support-closed-form} may change
when $S$ is replaced by a smaller sub-dcpo.  Accordingly, no monotonicity
of the assignment $S\mapsto\Supp_D(S)$ is asserted or needed.  The
maximality of $S_\infty(D)$ follows instead from the explicit transfinite
induction in \cref{thm:maximal-admissible}.  The transfinite induction, rather than monotonicity of \(S\mapsto\Supp_D(S)\), is therefore the basis of the maximality argument.
\end{remark}
\section{The Keimel-Lawson conditions in DCPO}\label{sec:KL}

This section is motivated by the Keimel--Lawson approach to
reflections in $\mathbf{Top}_0$, where suitable closure conditions on
a full subcategory containing the sober spaces lead to reflectivity;
see \cite{Keimel2009,Shen2024,Ershov2022}.  We establish an analogous criterion in $\mathbf{DCPO}$.

By \cref{cor:minimal}, every non-discrete reflective full subcategory of
$\mathbf{DCPO}$ contains $R_2$.  Hence, for a full subcategory
\(\mathcal K\) with $R_2\subseteq \mathcal K$, it is natural to construct the
\(\mathcal K\)-reflection of $D$ inside its $R_2$-reflection $\mathbf{L_2}D$.  We do this
by considering \(\mathcal K\)-subspaces of $\mathbf{L_2}D$ containing $\lambda_D[D]$.
\cref{thm:KL} characterizes when the resulting hull yields the
\(\mathcal K\)-reflection.

\subsection{Scott subspaces}

\begin{definition}
Let \(A\subseteq X\) be a sub-dcpo, with inclusion \(i_A:A\hookrightarrow X\). We call \(A\) a \emph{Scott subspace} of \(X\), and write \(A\le_\Sigma X\), if
\[
  i_A^{-1}:\sigma(X)\twoheadrightarrow\sigma(A)
\]
is surjective.
\end{definition}

Only surjectivity is required: it says that the Scott topology of \(A\) is the subspace topology inherited from \(X\). Requiring \(i_A^{-1}\) to be an isomorphism would impose the much stronger condition that distinct Scott-open subsets of \(X\) remain distinct after restriction to \(A\).

Let \(\mathcal K\subseteq\DCPO\) be a full subcategory closed under isomorphisms, and let \(X\in\Rtwo\). If \(A\le_\Sigma X\) and \(A\in\mathcal K\), we write \(A\le_\mathcal KX\). If the intersection of all \(\mathcal K\)-subspaces of \(X\) containing a subset \(S\subseteq X\) is again a \(\mathcal K\)-subspace, we denote it by
\[
  \langle S\rangle_X^{\mathcal K}
\]
and call it the \(\mathcal K\)-subspace hull of \(S\) in \(X\).

\begin{lemma}\label{lem:restriction-unique}
Let \(H\le_\Sigma \mathbf{L_2}D\), with inclusion \(j:H\hookrightarrow \mathbf{L_2}D\), and let \(\kappa:D\to H\) satisfy \(j\circ\kappa=\lambda_D\). If \(h_1,h_2:H\to Z\) satisfy \(h_1\circ\kappa=h_2\circ\kappa\), then \(h_1=h_2\).
\end{lemma}

\begin{proof}
Let \(W\in\sigma(Z)\). Since \(H\le_\Sigma \mathbf{L_2}D\), there exist \(V_i\in\sigma(\mathbf{L_2}D)\) such that
\[
  h_i^{-1}(W)=j^{-1}(V_i),
  \  i=1,2.
\]
Then
\[
  \lambda_D^{-1}(V_i)=\kappa^{-1}h_i^{-1}(W).
\]
The assumption \(h_1\circ\kappa=h_2\circ\kappa\), together with injectivity of \(\lambda_D^{-1}\), gives \(V_1=V_2\). Hence \(h_1^{-1}(W)=h_2^{-1}(W)\) for every Scott-open \(W\), and the \(T_0\) property implies \(h_1=h_2\).
\end{proof}

\subsection{The Keimel-Lawson conditions}

\begin{theorem}\label{thm:KL}
Let \(\mathcal K\subseteq\DCPO\) be a full subcategory. Then \(\mathcal K\) is a non-discrete reflective full subcategory of \(\DCPO\) if and only if the following conditions hold.
\begin{enumerate}[label=\textup{(KL\arabic*)}]
\item \(\Rtwo\subseteq\mathcal K\).
\item For every dcpo \(D\), the hull
\( \langle\lambda_D[D]\rangle_{\mathbf{L_2}D}^{\mathcal K}
\) exists.
\item For every \(Y\in\mathcal K\), every dcpo \(D\), and every Scott-continuous map \(g:\mathbf{L_2}D\to \mathbf{L_2}Y\),
\[
  g(\lambda_D[D])\subseteq\lambda_Y[Y]
\]
implies
\[
  \langle\lambda_D[D]\rangle_{\mathbf{L_2}D}^{\mathcal K}
  \subseteq g^{-1}(\lambda_Y[Y]).
\]
\end{enumerate}
\end{theorem}

\begin{proof}
Assume first that \emph{(KL1)--(KL3)} hold. Put
\[
  K_D=\langle\lambda_D[D]\rangle_{\mathbf{L_2}D}^{\mathcal K},
\]
and let \(j_D:K_D\hookrightarrow \mathbf{L_2}D\) be the inclusion. Since \(\lambda_D[D]\subseteq K_D\), there is a map \(\kappa_D:D\to K_D\) with
\(j_D\circ\kappa_D=\lambda_D\). It is Scott continuous. Indeed, if
\(W\in\sigma(K_D)\), then \(W=j_D^{-1}(V)\) for some
\(V\in\sigma(\mathbf{L_2}D)\), and
\(\kappa_D^{-1}(W)=\lambda_D^{-1}(V)\).

Let \(Y\in\mathcal K\) and \(f:D\to Y\). The \(\Rtwo\)-reflection gives a unique Scott-continuous map \(\bar f:\mathbf{L_2}D\to \mathbf{L_2}Y\) satisfying
\[
  \bar f\circ\lambda_D=\lambda_Y\circ f.
\]
Thus \(\bar f(\lambda_D[D])\subseteq\lambda_Y[Y]\). By \emph{(KL3)},
\[
  \bar f(K_D)\subseteq\lambda_Y[Y].
\]
The image \(\lambda_Y[Y]\) is a sub-dcpo of \(\mathbf{L_2}Y\), so the
corestriction \(\bar f|_{K_D}:K_D\to\lambda_Y[Y]\) is Scott
continuous. Moreover, \cref{lem:open-embedding} shows that
\(\lambda_Y:Y\to\lambda_Y[Y]\) is a dcpo isomorphism. Hence there is a
unique Scott-continuous map \(\widetilde f:K_D\to Y\) such that
\(\lambda_Y\circ\widetilde f=\bar f|_{K_D}\).
Injectivity of \(\lambda_Y\) gives \(\widetilde f\circ\kappa_D=f\), and uniqueness follows from \cref{lem:restriction-unique}.

We must also verify the enriched condition. If \(f\le g\) in \([D\to Y]\), then reflectivity of \(\mathbf{L_2}\) gives \(\bar f\le\bar g\). Restriction to \(K_D\) preserves this inequality, and the inverse of the order isomorphism \(\lambda_Y:Y\to\lambda_Y[Y]\) is monotone. Hence
\[
  \widetilde f\le\widetilde g.
\]
Thus the unique-extension operator \([D\to Y]\to[K_D\to Y]\) is order-preserving, so
\[
  \kappa_D^*:[K_D\to Y]\cong[D\to Y]
\]
is an isomorphism of dcpos.

Conversely, assume that \(\mathcal K\) is a non-discrete enriched reflective full subcategory of \(\DCPO\), with unit \(\kappa_D:D\to K_D\). By \cref{lem:retract} and closure under retracts, \(2\in\mathcal K\). Hence \(\Rtwo\subseteq\mathcal K\) by \cref{thm:replete-hull}, proving \emph{(KL1)}.

Since \(\mathbf{L_2}D\in\mathcal K\), the map \(\lambda_D\) factors uniquely as
\[
  D\xrightarrow{\kappa_D}K_D\xrightarrow{m_D}\mathbf{L_2}D.
\]
Because \(2\in\mathcal K\), we have an isomorphism of dcpos
\[
  \kappa_D^{-1}:\sigma(K_D)\cong\sigma(D).
\]
Together with \(\lambda_D^{-1}=\kappa_D^{-1}\circ m_D^{-1}\), this
implies that \(m_D^{-1}\) is an order isomorphism. By
\cref{lem:open-embedding}, \(m_D:K_D\to m_D[K_D]\) is a dcpo
isomorphism. Since \(\mathcal K\) is replete, \(m_D[K_D]\in\mathcal K\).
Moreover, the surjectivity of \(m_D^{-1}:\sigma(\mathbf{L_2}D)\to\sigma(K_D)\),
together with the isomorphism \(K_D\cong m_D[K_D]\), shows that
\(m_D[K_D]\le_\Sigma \mathbf{L_2}D\). Thus \(m_D[K_D]\) is a
\(\mathcal K\)-subspace of \(\mathbf{L_2}D\).

Let \(A\le_\mathcal K \mathbf{L_2}D\) and suppose \(\lambda_D[D]\subseteq A\). Then \(\lambda_D\) factors through \(A\). By reflectivity, there is a map \(u:K_D\to A\), and uniqueness gives
\[
  i_A\circ u=m_D.
\]
Hence \(m_D[K_D]\subseteq A\). Thus
\[
  m_D[K_D]
  =\langle\lambda_D[D]\rangle_{\mathbf{L_2}D}^{\mathcal K},
\]
which proves \emph{(KL2)}.

Finally, let \(g:\mathbf{L_2}D\to \mathbf{L_2}Y\) satisfy
\[
  g(\lambda_D[D])\subseteq\lambda_Y[Y].
\]
The map \(g\circ\lambda_D\) has image in \(\lambda_Y[Y]\). Its
corestriction to \(\lambda_Y[Y]\) is Scott continuous, so define
\(f=\lambda_Y^{-1}\circ(g\circ\lambda_D):D\to Y\). Extend \(f\) uniquely
to a Scott-continuous map \(\widetilde f:K_D\to Y\). The maps \(g\circ m_D\) and \(\lambda_Y\circ\widetilde f\) agree after precomposition with \(\kappa_D\), so uniqueness of the \(\mathcal K\)-reflection gives
\[
  g\circ m_D=\lambda_Y\circ\widetilde f.
\]
Therefore
\[
  g(m_D[K_D])\subseteq\lambda_Y[Y],
\]
which is exactly \emph{(KL3)}.
\end{proof}

\begin{figure}[ht]
\centering
\begin{tikzcd}[column sep=large,row sep=large]
D \arrow[r,"\kappa_D"] \arrow[d,"\lambda_D"']
  & K_D \arrow[d,"m_D"] \arrow[dr,dashed,"\widetilde f"] & \\
\mathbf{L_2}D \arrow[r,"g"'] & \mathbf{L_2}Y & Y \arrow[l,"\lambda_Y"']
\end{tikzcd}
\caption{The central diagram in the relative hull criterion.}
\label{fig:KL}
\end{figure}

\section*{Acknowledgment}
During the preparation of this manuscript, the authors used AI-assisted tools for language polishing and grammar checking. The authors carefully reviewed and verified the final manuscript and take full responsibility for its content, including the correctness of all mathematical statements, proofs, and references.


\begin{thebibliography}{99}

\bibitem{bks2014}
I.~Battenfeld, K.~Keimel, and T.~Streicher,
\newblock Observationally-induced algebras in domain theory,
\newblock \emph{Logical Methods in Computer Science} \textbf{10}(3:18), 2014.
DOI: \href{https://doi.org/10.2168/LMCS-10(3:18)2014}{10.2168/LMCS-10(3:18)2014}.

\bibitem{Ershov2022}
Y.~Ershov,
\newblock K-completions of T0 spaces,
\newblock \emph{Algebra and Logic} \textbf{61}(3), 177--187, 2022.

\bibitem{ershov1999}
Y.~L.~Ershov,
\newblock On d-spaces,
\newblock \emph{Theoretical Computer Science} \textbf{224} (1999), 59--72.

\bibitem{gierz2003}
G.~Gierz, K.~H.~Hofmann, J.~D.~Lawson, M.~Mislove, and D.~S.~Scott,
\newblock \emph{Continuous Lattices and Domains},
\newblock Encyclopedia of Mathematics and its Applications, vol.~93,
Cambridge University Press, Cambridge, 2003.

\bibitem{goubault2013}
J.~Goubault-Larrecq,
\newblock \emph{Non-Hausdorff Topology and Domain Theory: Selected Topics in Point-Set Topology},
\newblock New Mathematical Monographs, vol.~22,
Cambridge University Press, Cambridge, 2013.

\bibitem{HoZhao2009}
W.~K.~Ho and D.~Zhao,
\newblock Lattices of Scott-closed sets,
\newblock \emph{Commentationes Mathematicae Universitatis Carolinae}
\textbf{50}(2) (2009), 297--314.

\bibitem{ho2018}
W.~K.~Ho, J.~Goubault-Larrecq, A.~Jung, and X.~Xi,
\newblock The Ho--Zhao problem,
\newblock \emph{Logical Methods in Computer Science}
\textbf{14}(1:7), 2018.
DOI: \href{https://doi.org/10.23638/LMCS-14(1:7)2018}{10.23638/LMCS-14(1:7)2018}.

\bibitem{Keimel2009}
K.~Keimel and J.~D.~Lawson,
\newblock D-completions and the d-topology,
\newblock \emph{Annals of Pure and Applied Logic}
\textbf{159}(3), 292--306, 2009.

\bibitem{miao2024}
H.~Miao, H.~Hou, X.~Jia, and Q.~Li,
\newblock The category of well-filtered dcpos is not \(\Gamma\)-faithful,
\newblock arXiv:2409.01546, 2024.

\bibitem{Riehl2017}
E.~Riehl,
\newblock \emph{Category Theory in Context},
\newblock Courier Dover Publications, Mineola, 2016.

\bibitem{sc-4}
D.~S.~Scott,
\newblock Domains for denotational semantics,
\newblock in: M.~Nielsen and E.~M.~Schmidt (eds.),
\emph{Automata, Languages and Programming},
LNCS 140, Springer, 1982, 577--613.
DOI: \href{https://doi.org/10.1007/BFb0012801}{10.1007/BFb0012801}.

\bibitem{Shen2019}
C.~Shen, X.~Xi, X.~Xu, and D.~Zhao,
\newblock On well-filtered reflections of T0 spaces,
\newblock \emph{Topology and its Applications}
\textbf{267}, 106869, 2019.

\bibitem{Shen2024}
C.~Shen, X.~Xi, and D.~Zhao,
\newblock The reflectivity of some categories of T0 spaces in domain theory,
\newblock \emph{Rocky Mountain Journal of Mathematics}
\textbf{54}(4), 1149--1166, 2024.

\bibitem{Wu2019}
G.~Wu, X.~Xi, X.~Xu, and D.~Zhao,
\newblock Existence of well-filterification,
\newblock \emph{Topology and its Applications}
\textbf{267}, 107044, 2019.

\bibitem{Wyler1981}
U.~Wyler,
\newblock Dedekind complete posets and Scott topologies,
\newblock in: Lecture Notes in Mathematics, vol.~871,
Springer-Verlag, Berlin, 1981, 384--389.

\bibitem{Xu2020}
X.~Xu,
\newblock A direct approach to k-reflections of T0 spaces,
\newblock \emph{Topology and its Applications}
\textbf{272}, 107076, 2020.

\bibitem{XuZhao2019}
L.~Xu and D.~Zhao,
\newblock $C_{\sigma}$-unique dcpos and non-maximality of the class of
dominated dcpos regarding \(\Gamma\)-faithfulness,
\newblock \emph{Electronic Notes in Theoretical Computer Science}
\textbf{345} (2019), 249--260.
DOI: \href{https://doi.org/10.1016/j.entcs.2019.07.027}{10.1016/j.entcs.2019.07.027}.

\bibitem{ZhaoXu2018}
D.~Zhao and L.~Xu,
\newblock Uniqueness of directed complete posets based on Scott closed set lattices,
\newblock \emph{Logical Methods in Computer Science}
\textbf{14}(2:10) (2018), 1--12.

\end{thebibliography}
\end{document}